\documentclass[11pt,a4paper]{article}
\usepackage{amsmath}
\usepackage[a4paper]{geometry}
\usepackage{amssymb,latexsym,amsmath,amsfonts,amsthm}
\usepackage{graphicx}
\usepackage{algorithm}
\usepackage{algorithmic}
 \usepackage{dsfont}
\usepackage{mathrsfs}
\usepackage{epsfig}
\usepackage{booktabs}
\usepackage{mathtools}
\usepackage{comment,verbatim}
\usepackage{overpic}
\usepackage{hyperref}
\usepackage{bbm}
\usepackage[toc,page]{appendix}

\newtheorem{conjecture}{Conjecture}[section]

\newtheorem{theorem}{Theorem}[section]
\newtheorem{lemma}[theorem]{Lemma}
\newtheorem{proposition}[theorem]{Proposition}

\newtheorem{corollary}[theorem]{Corollary}

\theoremstyle{definition}

\theoremstyle{definition}

\theoremstyle{remark}

\newtheorem{remark}[theorem]{Remark}

\numberwithin{equation}{section}

\usepackage{color}

\usepackage[normalem]{ulem}

\def\XXint#1#2#3{{
\setbox0=\hbox{$#1{#2#3}{\int}$}
\vcenter{\hbox{$#2#3$}}\kern-.5\wd0}}

\DeclarePairedDelimiter\floor{\lfloor}{\rfloor}

\begin{document}
\title{Jacobi polynomials on the Bernstein ellipse} %\footnotemark[1]}
\author{Haiyong Wang\footnotemark[1]~\footnotemark[2]~ and Lun Zhang\footnotemark[3] }

\maketitle
\renewcommand{\thefootnote}{\fnsymbol{footnote}}

\footnotetext[1]{School of Mathematics and Statistics, Huazhong
University of Science and Technology, Wuhan 430074, P. R. China.
E-mail: \texttt{haiyongwang@hust.edu.cn}}

\footnotetext[2]{Hubei Key Laboratory of Engineering Modeling and
Scientific Computing, Huazhong University of Science and Technology,
Wuhan 430074, P. R. China.}

\footnotetext[3]{School of Mathematical Sciences and Shanghai Key
Laboratory for Contemporary Applied Mathematics, Fudan University,
Shanghai 200433, P. R. China. E-mail:
\texttt{lunzhang@fudan.edu.cn}}

\begin{abstract}
In this paper, we are concerned with Jacobi polynomials
$P_n^{(\alpha,\beta)}(x)$ on the Bernstein ellipse with motivation
mainly coming from recent studies of convergence rate of spectral
interpolation. An explicit representation of
$P_n^{(\alpha,\beta)}(x)$ is derived in the variable of
parametrization. This formula further allows us to show that the
maximum value of $\left|P_n^{(\alpha,\beta)}(z)\right|$ over the
Bernstein ellipse is attained at one of the endpoints of the major
axis if $\alpha+\beta\geq -1$. For the minimum value, we are able to
show that for a large class of Gegenbauer polynomials (i.e.,
$\alpha=\beta$), it is attained at two endpoints of the minor axis.
These results particularly extend those previously known only for
some special cases. Moreover, we obtain a more refined asymptotic
estimate for Jacobi polynomials on the Bernstein ellipse.

\end{abstract}


\vspace{0.05in}


\section{Introduction}\label{sec:introduction}
Spectral collocation method is a classical and powerful tool to
solve integral and differential equations.
%\cite{hussaini2006spectral,shen2011spectral,trefethen2000spectral}.
Suppose that the equation is defined on a finite interval $[-1,1]$,
the basic idea of this approach is to approximate the solution of
the equation by its polynomial interpolant of the form
\begin{align}
f(x) \approx p_{n}(x) = \sum_{k=1}^{n} f(x_{k}) \ell_k(x), \qquad -1
\leq x \leq 1,
\end{align}
where $\{ x_{k} \}_{k=1}^{n}$ is a set of distinct nodes and
\[
\ell_k(x) = \prod_{\substack{j=1 \\ j\neq k}}^{n} \frac{x -
x_{j}}{x_{k} - x_{j}}, \qquad 1 \leq k \leq n,
\]
are the Lagrange fundamental polynomials. The values $\{ f(x_k)
\}_{k=1}^{n}$ are determined by requiring that the interpolant
$p_n(x)$ satisfies the equation exactly at the nodes $\{ x_k
\}_{k=1}^{n}$. To ensure rapid convergence of the spectral
collocation method, the interpolation nodes $\{ x_k \}_{k=1}^{n}$
with the distribution of density $(1-x^2)^{-1/2}$ are preferable,
and the ideal candidates are the zeros or extrema of classical
orthogonal polynomials such as Gegenbauer polynomials, or more
generally, the Jacobi polynomials; cf.
\cite{hussaini2006spectral,shen2011spectral,trefethen2000spectral}.
The interpolation procedure described above is also known as
spectral interpolation.

As is well known, the accuracy of spectral interpolation depends on
the regularity of the underlying function $f(x)$, with exponential
rate if $f(x)$ is analytic in a neighborhood containing the interval
$[-1,1]$; we refer to
\cite{reddy2005accuracy,tadmor1986exponential,wang2012convergence,wang2014superconvergence,xie2013exponential,xiang2010error,zhang2012superconvergenc}
for relevant results and \cite{wang2014explicit} for fast
implementation.
%rigorous proofs of the exponential convergence rate.
To this end, it is worthwhile to recall that the starting point of
these proofs is the so-called Hermite integral formula. More
precisely, let $\mathcal{E}_{\rho}$ be the Bernstein ellipse:
\begin{equation}\label{def:Bernstein}
 \mathcal{E}_{\rho} =  \left\{ z \in \mathbb{C} ~~\bigg| ~~z =
 \frac{1}{2} \left( u + u^{-1} \right),~~ u = \rho e^{i \theta},~~ \rho\geq1,~  ~0 \leq
\theta  < 2\pi  \right\}.
\end{equation}
The Bernstein ellipse $\mathcal{E}_{\rho}$ has the foci at $\pm 1$
with the major and minor semi-axes given by
$\frac{1}{2}(\rho+\rho^{-1})$ and $\frac{1}{2}(\rho-\rho^{-1})$,
respectively. Suppose that $f(x)$ is analytic on and within
$\mathcal{E}_{\rho}$ for some $\rho>1$, it follows from the Hermite
integral formula \cite[Theorem 3.6.1]{davis1975interpolation} that
\begin{align}\label{eq:hermite integral formula}
f(x) - p_n(x) = \frac{1}{2\pi i} \oint_{\mathcal{E}_{\rho}}
\frac{\omega_n(x) f(z)}{\omega_n(z) (z-x)} dz,
\end{align}
where $$\omega_n(x) = d_n (x-x_1)(x-x_2)\cdots(x-x_n)$$ with $d_n$ being a
positive normalization constant.
This in turn implies that
\begin{align}\label{eq:bound for interpolant}
| f(x) - p_n(x) | \leq \frac{M L(\mathcal{E}_{\rho} )}{2\pi d}
\max_{\substack{x\in[-1,1] \\ z \in \mathcal{E}_{\rho} }} \left|
\frac{ \omega_n(x) }{ \omega_n(z) } \right|,
\end{align}
where $M = \max_{z\in \mathcal{E}_{\rho}} |f(z)|$,
$L(\mathcal{E}_{\rho} )$ denotes the length of the circumference of
$\mathcal{E}_{\rho}$, and $d$ is the distance from
$\mathcal{E}_{\rho}$ to the interval $[-1,1]$.
%After a suitable
%normalization on $\max_{x\in[-1,1]} |\omega_n(x)|$, it is clear that
%the error estimate follows from an estimate of the polynomial
%$\omega_n(z)$ on the Bernstein ellipse.

For polynomial interpolation at the Jacobi points, i.e., the nodes
$\{x_k\}_{k=1}^n$ are the roots of $n$-th Jacobi polynomial
$P_{n}^{(\alpha,\beta)}(x)$ ($\alpha,\beta>-1$) and the polynomial
$\omega_n(z)$ in \eqref{eq:hermite integral formula} is taken to be
$P_{n}^{(\alpha,\beta)}(z)$, the properties of $P_{n}^{(\alpha,\beta)}(x)$ on the Bernstein ellipse are essential in the
analysis of convergence rate. For the Gegenbauer polynomials $C_n^{\lambda}(z)$
(corresponding to the special cases
$\alpha=\beta=\lambda-\frac{1}{2}$ of the Jacobi polynomials), by
noting that (see \cite[Lemma 3.1]{xie2013exponential})
\begin{equation}\label{eq:XWZexpli}
C_n^{\lambda}(z)=\sum_{k=0}^n g_{k}^\lambda g_{n-k}^\lambda
u^{n-2k},  \quad z=\frac{1}{2}\left(u+u^{-1}\right), \quad n\geq 0,
\end{equation}
where
$$g_0^\lambda=1, \qquad g_k^\lambda=\binom{k+\lambda-1}{k}=\frac{\Gamma(k+\lambda)}{k!\Gamma(\lambda)}, \qquad 1\leq k\leq n,$$
the following asymptotic estimate of Gegenbauer polynomials on the
Bernstein ellipse is obtained by Xie, Wang and Zhao in \cite[Theorem
3.2]{xie2013exponential}: there exists $0<\varepsilon \leq
\frac{1}{2}$ such that
\begin{align}\label{eq:estXWZ}
\left| (1 - u^{-2})^{-\lambda} -
\frac{C_n^{\lambda}(z)}{g_n^{\lambda} u^n} \right| \leq
A(\rho,\lambda) n^{\varepsilon - 1} + \mathcal{O}(n^{-1}), \quad
z\in\mathcal{E}_{\rho},
\end{align}
for $\rho>1$, $\lambda>-\frac{1}{2}$ and $\lambda\neq0$, where
\[
A(\rho,\lambda) = |1-\lambda| \left| (1-\rho^{-2})^{-\lambda} - 1
\right|.
\]
The estimate \eqref{eq:estXWZ} plays an important role in the
rigorous proofs of exponential convergence of Gegenbauer
interpolation and spectral differentiation conducted in
\cite{xie2013exponential}. Later, in a paper regarding
superconvergence of Jacobi-Gauss type spectral interpolation
\cite{wang2014superconvergence}, Wang, Zhao and Zhang have made use
of the following estimate of the lower bound for Jacobi polynomial
on the Bernstein ellipse:
\begin{align}\label{eq:WZZ}
\min_{z\in \mathcal{E}_{\rho}} \left|P^{(\alpha,\beta)}_n (z)\right| \geq
C(\rho;\alpha,\beta) n^{-\frac{1}{2}} \rho^{n+1} ( 1 + \mathcal{O}(n^{-1})),
\end{align}
where $C(\rho;\alpha,\beta) = \min_{|u|=\rho}
|\phi_0(u;\alpha,\beta)|$ is a constant independent of $n$, and the
function $\phi_0(u;\alpha,\beta)$ is regular for $|u|=\rho>1$, and
$|u|=1$ but $u\neq \pm 1$; see
\cite[Equation~(4.7)]{wang2014superconvergence}. This result follows
directly from the asymptotic formula of Jacobi polynomials
\cite[Theorem~8.21.9]{szego1939orthogonal}. We note that, however,
except for the very special cases like $\alpha=\beta=-\frac{1}{2}$,
the explicit form of $C(\rho;\alpha,\beta)$ is not available.

We also note that there is a close connection between polynomial
interpolation and the potential theory \cite[Chapter
5]{trefethen2000spectral}. More specifically, let us define the
discrete potential function associated with the nodes
$\{x_k\}_{k=1}^{n}$ by
\begin{align*}
E_n(z) = \frac{1}{n} \sum_{k=1}^{n} \log|z - x_k|.
\end{align*}
This function is harmonic in the complex plane except at
$\{x_k\}_{k=1}^{n}$ and can be viewed as the potential generated by
all $\{x_k\}_{k=1}^{n}$ if each $x_k$ is interpreted as a point
charge of strength $1/n$ and the repulsion is inverse-linear.
Clearly, we have $|\omega_n(z)|=e^{n E_n(z)}$. Thus, if we choose
$x_k$ to be the zeros of $P_n^{(\alpha,\beta)}(x)$, then the extrema
of $|P_n^{(\alpha,\beta)}(z)|$ implies the extrema of the
corresponding potential $E_n(z)$ as well.

It is the aim of the present research to conduct more complete
studies of Jacobi polynomials on the Bernstein ellipse, including
the explicit formula, extrema of the absolute value and the
asymptotic estimate. Our main contributions are listed below:
\begin{itemize}
\item An explicit formula of $P_n^{(\alpha,\beta)}(x)$ is derived in the variable of parametrization, which generalizes \eqref{eq:XWZexpli} valid for the Gegenbauer case. %It comes out that the coefficients therein involve the hypergeometric function ${}_3\mathrm{F}_2$, and a three-term recurrence relation for these coefficients is also presented for efficient computations.
\item The extrema of $\left|P_n^{(\alpha,\beta)}(z)\right|$ on the Bernstein ellipse $\mathcal{E}_{\rho}$ are identified under some assumptions on the parameters. We show that the maximum value is attained at one of the endpoints of the major axis if $\alpha+\beta\geq -1$. This particularly extends \cite[Theorem 4.5.1]{ismail1998classical} established by Ismail, which is valid for the Gegenbauer polynomials. For the minimum value, we are able to show that for a large class of Gegenbauer polynomials, it is attained at two endpoints of the minor axis.
\item We provide a more refined and computable asymptotic estimate as well as a lower bound for the Jacobi polynomials on the Bernstein ellipse, which generalizes \eqref{eq:estXWZ} concerning the Gegenbauer polynomials. % and implies a more explicit expression for the constant $C(\rho,\alpha,\beta)$ in \eqref{eq:WZZ}.
\end{itemize}

%We note that, by setting $u_n(z):=\frac{1}{n}\log
%\left|P_n^{(\alpha,\beta)}(z)\right|$, $z\in \mathbb{C}\setminus
%[-1,1]$, the extrema of $\left|P_n^{(\alpha,\beta)}\right|$ on the
%Bernstein ellipse $\mathcal{E}_{\rho}$ also imply those of $u_n(z)$.
%There is an electrostatic interpretation of $u_n(z)$. If we think of
%each zero of $P_n^{(\alpha,\beta)}(x)$ as a point charge of strength
%$1/n$, then $u_n$ can be viewed as the potential generated by all
%these charges with inverse-linear repulsion.

The rest of this paper is organized as follows. We first give a
brief review of Jacobi polynomials in Section \ref{sec:property},
which includes some basic properties that will be used later.
Section \ref{sec:Jacobi ellipse} is devoted to the explicit
representation of Jacobi polynomials on the Bernstein ellipse. %in the $u$ variable.
%The idea is to expand Jacobi polynomials in terms of Chebyshev polynomials of the first kind $T_n(x)$, since $T_n(x)$ has a simple explicit formula on the Bernstein ellipse.
The extrema of Jacobi polynomials on the Bernstein ellipse are
discussed in Section \ref{sec:JacobiExtrema}. The identification of
maximum value relies on a three-term recurrence relation for the
coefficients arising in the explicit formula.
%, we are able to determine their signs and the maximum value follows.
For the minimum value, we first deal with the Chebyshev polynomials
of the first and second kinds and then extend the results to
Gegenbauer polynomials. The asymptotic estimate and the lower bound
of Jacobi polynomials on the Bernstein ellipse are presented in
Section \ref{sec:JacobiBound}. %Instead of estimating the
\section{Some properties of Jacobi polynomials}
\label{sec:property}
In this section, we collect some basic properties of Jacobi
polynomials which will be used in the subsequent analysis. All these
properties can be found in the classical book of Szeg\H{o} \cite{szego1939orthogonal}.

Let $P_{n}^{(\alpha,\beta)}(x)$ denote the Jacobi polynomial of
degree $n$, which is defined explicitly by
\begin{align}\label{def:jacob poly}
P_{n}^{(\alpha,\beta)}(x) = 2^{-n} \sum_{k=0}^{n}
\binom{n+\alpha}{n-k} \binom{n+\beta}{k} (x-1)^k (x+1)^{n-k}, \quad
\alpha,\beta>-1.
\end{align}
The Jacobi polynomials are orthogonal over $[-1,1]$ with respect to
the weight function $(1-x)^{\alpha} (1+x)^{\beta} $, that is,
\begin{align}
\int_{-1}^{1} (1-x)^{\alpha} (1+x)^{\beta} P_{n}^{(\alpha,\beta)}(x)
P_{m}^{(\alpha,\beta)}(x) dx =  h_n^{(\alpha,\beta)}\delta_{m,n},
\end{align}
where $\delta_{m,n}$ is the Kronecker delta and
\[
h_n^{(\alpha,\beta)} = \frac{2^{\alpha+\beta+1}}{2n+\alpha+\beta+1}
\frac{\Gamma(n+\alpha+1)\Gamma(n+\beta+1)}{\Gamma(n+\alpha+\beta+1)n!}.
\]
From the explicit formula \eqref{def:jacob poly}, it is easily seen that
\begin{align}\label{eq:JacPol}
P_{n}^{(\alpha,\beta)}(x) = k_n^{(\alpha,\beta)} x^n + \cdots,
\end{align}
where the leading coefficient $k_n^{(\alpha,\beta)}$ is given by
\begin{equation}\label{eq:kn}
 k_n^{(\alpha,\beta)} = \frac{1}{2^n} \binom{2n+\alpha+\beta}{n}=\frac{\Gamma(2n+\alpha+\beta+1)}{2^n n! \Gamma(n+\alpha+\beta+1)}.
\end{equation}
Let $q := \max\{\alpha,\beta\}$ with $\alpha,\beta>-1$. The maximum of $\left|P_{n}^{(\alpha,\beta)}(x)\right|$ on the interval $[-1,1]$ is given by (see \cite[Theorem 7.32.1]{szego1939orthogonal})
\begin{align}\label{eq:max}
\max_{x\in[-1,1]}\left|P_{n}^{(\alpha,\beta)}(x)\right| =
\left\{\begin{array}{cc}
                                          {\displaystyle \binom{n+q}{n}}, & \mbox{if $q \geq -\frac{1}{2}$}, \\ [15pt]
                                          {|P_{n}^{(\alpha,\beta)}(\tilde{x})|}, & \mbox{if $q <
                                          -\frac{1}{2}$},
                                        \end{array}
                                        \right.
\end{align}
where $\tilde{x}$ is one of the two maximum points nearest
$(\beta-\alpha)/(\alpha+\beta+1)$. Indeed, when $q \geq
-\frac{1}{2}$, the maximum of
$\left|P_{n}^{(\alpha,\beta)}(x)\right|$ is attained at one of the
endpoints $\{-1,1\}$. %and the estimate can be seen from the fact
%that $\binom{n+q}{n}=\frac{\Gamma(n+q+1)}{\Gamma(1+q)\Gamma(n+1)}$
%and the ratio asymptotics of Gamma functions (cf. \cite[Formula
%(5.11.13)]{DLMF})

When $\alpha=\beta$, the Jacobi polynomials are, up to some positive constants, also known as Gegenbauer
(or the ultraspherical) polynomials $C_n^{\lambda}(x)$. More precisely, we have
\begin{align}\label{eq:gegenbauerdef}
C_n^{\lambda}(x) =
\frac{\Gamma(\lambda+\frac{1}{2})}{\Gamma(2\lambda)}
\frac{\Gamma(n+2\lambda)}{\Gamma(n+\lambda+\frac{1}{2})}
P_n^{(\lambda-\frac{1}{2},\lambda-\frac{1}{2})}(x), \qquad
\lambda>-\frac{1}{2}.
\end{align}
The orthogonality of Gegenbauer polynomials reads
$$\int_{-1}^{1} C_m^{\lambda}(x) C_n^{\lambda}(x) (1-x^2)^{\lambda-1/2}dx = h_n^{\lambda} \delta_{m,n},$$
where
$h_n^{\lambda}=\frac{2^{1-2\lambda}\pi \Gamma(n+2\lambda)}{\Gamma^2(\lambda)n!(n+\lambda)}.$
Since the weight function $(1-x^2)^{\lambda-1/2}$ is an even function, it is readily seen
that following symmetry relations hold:
\begin{align}\label{eq:symmetry relation}
C_{n}^{\lambda}(x) = (-1)^n C_{n}^{\lambda}(-x), \quad n\geq 0.
\end{align}
Thus, $C_{n}^{\lambda}(x)$ is an even function for even $n$ and an
odd function for odd $n$.

The Chebyshev polynomials of the first and second kinds are
\begin{equation*}
T_n(\cos\theta) = \cos(n\theta), \qquad U_n(\cos\theta) = \frac{\sin((n+1)\theta)}{\sin\theta},\quad n\geq 0,
\end{equation*}
respectively. When $z\in \mathcal{E}_{\rho} $, the Chebyshev polynomials have the following simple representations in the variable of parametrization:
\begin{align}\label{eq:ChebyEllipse}
T_n(z) = \frac{1}{2}(u^n + u^{-n}), \qquad U_n(z) = \frac{u^{n+1} -
u^{-n-1}}{u - u^{-1}}.
\end{align}
They are special cases of Gegenbauer polynomials, and the relations are given by
\begin{equation}
T_n(x)=\lim_{\lambda \to 0}\frac{n}{2}\frac{C_n^{\lambda}(x)}{\lambda},~~n\geq 1;\qquad U_n(x)=C_n^1(x),~~n\geq 0.
\end{equation}
Equivalently, one has
\begin{align}\label{eq:Cheb}
 T_n(x)= \frac{\Gamma(n+1)\Gamma(\frac{1}{2})}{\Gamma(n+\frac{1}{2})} P_n^{(-\frac{1}{2},-\frac{1}{2})}(x),
\qquad U_n(x) =
\frac{\Gamma(n+2)\Gamma(\frac{3}{2})}{\Gamma(n+\frac{3}{2})}P_n^{(\frac{1}{2},\frac{1}{2})}(x).
\end{align}

Finally, let $1>x_1^\lambda>x_2^\lambda>\ldots>x_n^\lambda>-1$ be the zeros of Gegenbauer polynomials. By \cite[Theorem 6.21.1]{szego1939orthogonal}, it follows that
\begin{align}\label{eq:decrease}
\frac{\partial x_j^{\lambda}}{\partial \lambda} < 0, \qquad  j =
1,\ldots,\floor{n/2},
\end{align}
where $\floor{x}$ denotes the integer part of $x$. Thus, the
positive zeros of a Gegenbauer polynomial $C_n^{\lambda}(x)$
strictly decrease with respect to the parameter $\lambda$.

%and
%\begin{align}\label{eq:Cheb}
%C_k^{1}(x) = U_k(x), \quad \lim_{\lambda\rightarrow0^{+}}
%\lambda^{-1} C_{n}^{(\lambda)}(x) = \frac{2}{n} T_n(x), \quad n \geq
%1,
%\end{align}
%where $T_k(x)$ is the Chebyshev polynomial of the first kind of
%degree $n$ defined as $T_k(\cos\theta) = \cos(k\theta)$, and
%$U_k(x)$ is the Chebyshev polynomial of the second kind of degree
%$n$ defined as $U_k(\cos\theta) = \sin((k+1)\theta)/\sin\theta$.

%Since $C_0^{(\lambda)}(x) = 1$, throughout this paper, we will only
%consider the case $n\geq1$.  Note that
%$$C_n^{(\lambda)}(\overline{z})=\overline{C_n^{(\lambda)}(z)},$$
%this, together with \eqref{eq:symmetry relation}, implies that
%$|C_n^{(\lambda)}(z)|$ is symmetric with respect to both coordinate
%axes, i.e.,
%\[
%\left|C_{n}^{(\lambda)}(z)\right| =
%\left|C_{n}^{(\lambda)}(-z)\right|, \quad
%\left|C_{n}^{(\lambda)}(\bar{z})\right| =
%\left|C_{n}^{(\lambda)}(z)\right|.
%\]
%It then suffices to restrict the consideration in the first quarter
%of the ellipse, i.e., to the interval $\theta \in
%[0,\frac{\pi}{2}]$.
%

%----------------------------------------------------------------------------------------
\section{An explicit formula of Jacobi polynomials on the Bernstein ellipse}
\label{sec:Jacobi ellipse}

It is the aim of this section to prove the following theorem, which
gives an explicit representation of $P_n^{(\alpha,\beta)}(x)$ on the
Bernstein ellipse in the variable of parametrization.
\begin{theorem}\label{thm:jacobi ellipse}
For $z\in \mathcal{E}_{\rho}$, i.e.,
\begin{equation}\label{eq:para}
z=\frac{1}{2} \left( u + u^{-1} \right),~~ |u| = \rho \geq 1,
\end{equation}
we have
\begin{align}\label{eq:jacobi ellipse}
P_n^{(\alpha,\beta)}(z) = \sum_{k = -n}^{n} d_{|k|,n} u^k,
\end{align}
where the coefficients are given by
\begin{multline}\label{eq:dkn}
d_{k,n} = \frac{ (n+\alpha+\beta+1)_k (k+\alpha+1)_{n-k} }{ (n-k)!
2^{2k} \Gamma(k+1) }  \\ \times
{}_3\mathrm{F}_2\left[\begin{matrix} k-n, ~
n+k+\alpha+\beta+1,~ k+\frac{1}{2};&
\\   k+\alpha+1,~ 2k+1;  &\end{matrix} \hspace{-.25cm} 1
\right],
\end{multline}
for $0 \leq k \leq n$.
\end{theorem}

To show Theorem \ref{thm:jacobi ellipse}, we note that Chebyshev polynomials of the first kind $T_n(x)$ has a simple explicit formula \eqref{eq:ChebyEllipse} on the Bernstein ellipse, the strategy is then to expand Jacobi polynomials in terms
of $T_n(x)$. The connection formula between two different families of Jacobi polynomials are stated in the following lemma (see \cite[Theorem~7.1.1]{andrews1999special}).
\begin{lemma}\label{lem:Jacobi connection}
Assume that
\begin{align}
P_n^{(\alpha,\beta)}(x) &= \sum_{k=0}^{n} c_{n,k}^{}
P_k^{(\gamma,\delta)}(x).
\end{align}
Then the connection coefficients are given by
\begin{align}
c_{n,k} &= \frac{ (n+\alpha+\beta+1)_k (k+\alpha+1)_{n-k}
(2k+\gamma+\delta+1) \Gamma(k+\gamma+\delta+1) }{ (n-k)!
\Gamma(2k+\gamma+\delta+2) } \nonumber \\
&~~~ \times {}_3\mathrm{F}_2\left[\begin{matrix} k-n, ~
n+k+\alpha+\beta+1,~ k+\gamma+1;&
\\   k+\alpha+1,~ 2k+\gamma+\delta+2;  &\end{matrix} \hspace{-.25cm} 1
\right],
\end{align}
where
\begin{equation}\label{eq:pochammer}
(a)_0 =1, \quad (a)_k =
\frac{\Gamma(a+k)}{\Gamma(a)}=a(a+1)\cdots(a+k-1),\quad k\geq1,
\end{equation}
is the Pochhammer symbol and
\begin{align}
{}_3\mathrm{F}_2\left[\begin{matrix} a_1, ~ a_2,~ a_3 ;&
\\  b_1,~ b_2; &\end{matrix} \hspace{-.25cm} z
\right] = \sum_{k=0}^{\infty} \frac{ (a_1)_k (a_2)_k (a_3)_k }{
(b_1)_k (b_2)_k } \frac{z^k}{k!}
\end{align}
is the generalized hypergeometric
function.

\end{lemma}
%\begin{proof}
%See \cite[Theorem~7.1.1]{andrews1999special}.
%\end{proof}

With the aid of Lemma \ref{lem:Jacobi connection}, we are now ready to prove Theorem \ref{thm:jacobi ellipse}.

\paragraph{Proof of Theorem \ref{thm:jacobi ellipse}}
By taking $\gamma=\delta=-\frac{1}{2}$ in Lemma \ref{lem:Jacobi
connection}, it follows that
\begin{align}
P_n^{(\alpha,\beta)}(x) &= \sum_{k=0}^{n}{'} \frac{
(n+\alpha+\beta+1)_k (k+\alpha+1)_{n-k} 2 \Gamma(k+1) }{ (n-k)!
\Gamma(2k+1) } \nonumber \\
&~~~ \times {}_3\mathrm{F}_2\left[\begin{matrix} k-n, ~
n+k+\alpha+\beta+1,~ k+\frac{1}{2};&
\\   k+\alpha+1,~ 2k+1;  &\end{matrix} \hspace{-.25cm} 1
\right] P_k^{(-\frac{1}{2},-\frac{1}{2})}(x),
\end{align}
where the prime indicates that the first term of the sum should be
halved. This, together with the identity \eqref{eq:Cheb}, gives
\begin{align}\label{eq:jacobi and cheby}
P_n^{(\alpha,\beta)}(x) &= \sum_{k=0}^{n}{'} \frac{
(n+\alpha+\beta+1)_k (k+\alpha+1)_{n-k} 2 \Gamma(k+1) }{ (n-k)!
\Gamma(2k+1) }  \nonumber \\
&~~~ \times {}_3\mathrm{F}_2\left[\begin{matrix} k-n, ~
n+k+\alpha+\beta+1,~ k+\frac{1}{2};&
\\   k+\alpha+1,~ 2k+1;  &\end{matrix} \hspace{-.25cm} 1
\right] \frac{\Gamma(k+\frac{1}{2})}{\Gamma(k+1)\Gamma(\frac{1}{2})} T_k(x) \nonumber \\
& = \sum_{k=0}^{n}{'} \frac{ (n+\alpha+\beta+1)_k (k+\alpha+1)_{n-k}
}{
(n-k)! 2^{2k-1} \Gamma(k+1) } \nonumber \\
&~~~ \times {}_3\mathrm{F}_2\left[\begin{matrix} k-n, ~
n+k+\alpha+\beta+1,~ k+\frac{1}{2};&
\\   k+\alpha+1,~ 2k+1;  &\end{matrix} \hspace{-.25cm} 1
\right] T_k(x),
\end{align}
where we have made use of the the duplication formula (cf. \cite[Formula 5.5.5]{DLMF})
\begin{equation}\label{eq:dupli}
\mathop{\Gamma\/}\nolimits\!\left(2z\right)=\pi^{-1/2}2^{2z-1}\mathop{\Gamma\/%
}\nolimits\!\left(z\right)\mathop{\Gamma\/}\nolimits\!\left(z+\tfrac{1}{2}%
\right),\qquad 2z\neq 0,-1,-2,\ldots,
\end{equation}
in the second step.

%Recall that $T_k(z)$ has the following simple expression on the Bernstein
%ellipse:
%\begin{align*}
%T_k(z) = \frac{1}{2} (u^k + u^{-k} ), \quad k\geq 0, \quad z\in \mathcal{E}_{\rho}.
%\end{align*}
Note that $T_k(z)$ has a simple expression $\frac{1}{2}(u+u^{-n})$ on $\mathcal{E}_{\rho}$.
Substituting this formula into the last equation of \eqref{eq:jacobi and cheby} gives us the desired result.

This completes the proof of Theorem \ref{thm:jacobi ellipse}. \qed
%\end{proof}

\begin{remark}
Suppose that $u=e^{i\theta}$ (i.e., $\rho=1$), we obtain from
\eqref{eq:jacobi ellipse} the following trigonometric
representations of Jacobi polynomials
\begin{equation}
P_n^{(\alpha,\beta)}(\cos\theta) = d_{0,n} + 2\sum_{k=1}^{n} d_{k,n}
\cos(k\theta).
\end{equation}
The above formula seems to be new, except for the special case $\alpha=\beta$ (cf. \cite[Formula (4.9.19)]{szego1939orthogonal}).
\end{remark}

\begin{remark}\label{rk:ultraspherical}
When $\alpha = \beta$, the coefficients $d_{k,n}$ can be further
simplified with the help of the properties of hypergeometric
function ${}_3\mathrm{F}_2$. Indeed, on account of the fact (see \cite[Theorem~3.5.5]{andrews1999special}) that
\begin{equation}
{}_3\mathrm{F}_2\left[\begin{matrix} a, ~
b,~ c;&
\\   (a+b+1)/2,~ 2c;  &\end{matrix} \hspace{-.25cm} 1
\right]=\frac{\Gamma\left(\frac{1}{2}\right)\Gamma\left(c+\frac{1}{2}\right)\Gamma\left(\frac{a+b+1}{2}\right)\Gamma\left(c-\frac{a+b-1}{2}\right)}
{\Gamma\left(\frac{a+1}{2}\right)\Gamma\left(\frac{b+1}{2}\right)\Gamma\left(c-\frac{a-1}{2}\right)\Gamma\left(c-\frac{b-1}{2}\right)},
\end{equation}
it follows from \eqref{eq:dkn} and straightforward calculations that
\begin{align}\label{eq:GegenbauerCoeff}
d_{k,n} = \left\{\begin{array}{cc}
  {\displaystyle \frac{2^{2\alpha} \Gamma(n+\alpha+1) \Gamma(\frac{k+n+1}{2}+\alpha) \Gamma(\frac{n-k+1}{2}+\alpha)}{\sqrt{\pi} \Gamma(n+2\alpha+1) \Gamma(\frac{k+n}{2}+1) \Gamma(\frac{n-k}{2}+1) \Gamma(\alpha+\frac{1}{2}) }  }, & \mbox{if $n-k$ is even}, \\ [15pt]
  { 0 }, & \mbox{if $n-k$ is odd}.
  \end{array}
  \right.
\end{align}
This particularly implies that
$$P_n^{(\alpha,\alpha)}(z) =
\sum_{k=0}^{n} d_{|n-2k|,n} u^{n-2k}.$$
Up to some constant factors, this recovers \eqref{eq:XWZexpli} which was derived via the three-term recurrence relation of Gegenbauer polynomials in \cite{xie2013exponential}. The approach used therein, however, seems difficult to be generalized to handle the Jacobi case.
\end{remark}

\section{Extrema of Jacobi polynomials on the Bernstein ellipse}
\label{sec:JacobiExtrema}

In this section, we will consider the extrema of Jacobi polynomials
on the Bernstein ellipse. The maximum value and the minimum value
will be discussed in subsections \ref{sec:max} and \ref{sec:min},
respectively.
%We can prove that the maximum value is attained at the intersection point of the real axis and the Bernstein ellipse under the assumption $\alpha+\beta\geq-1$. For the minimum, however, we can only prove the case that
%$\alpha=\beta\geq \frac{1}{2}$ and the special case $\alpha=\beta=-\frac{1}{2}$.

\subsection{Maximum value}\label{sec:max}
By \cite[Theorem 4.5.1]{ismail1998classical}, it is known that for the Gegenbauer polynomials $C_n^{\lambda}(x)$,
$\max\left|C_{n}^{\lambda}(z)\right|$, $z\in \mathcal{E}_{\rho}$ with $\lambda\geq 0$ is attained at the right endpoint of the major axis. It comes out that similar property holds for the Jacobi polynomials $P_n^{(\alpha,\beta)}$ with $\alpha+\beta \geq -1$, which is our main result of this section.

\begin{theorem}\label{prop:location}
For $\rho\geq1$ and $n\geq 1$, we have
\begin{enumerate}
\item[\rm (i)] If $\alpha>\beta$ and $\alpha+\beta\geq-1$, then the maximum value of $\left|P_{n}^{(\alpha,\beta)}(z)\right|$ is attained uniquely at the right endpoint of the major axis, i.e.,
\begin{equation}\label{eq:maxcaseI}
\max_{z\in \mathcal{E}_{\rho}} \left|P_{n}^{(\alpha,\beta)}(z)\right| =
P_{n}^{(\alpha,\beta)}\left(\tfrac{1}{2}\left(\rho+\rho^{-1}\right)\right).
\end{equation}

\item[\rm (ii)] If $\alpha<\beta$ and $\alpha+\beta\geq-1$, then the maximum value of $\left|P_{n}^{(\alpha,\beta)}(z)\right|$ is attained uniquely at the left endpoint of the major axis, i.e.,
\begin{equation}\label{eq:maxCaseII}
\max_{z\in \mathcal{E}_{\rho}} \left|P_{n}^{(\alpha,\beta)}(z)\right| = \left|P_{n}^{(\alpha,\beta)}\left(-\tfrac{1}{2}\left(\rho+\rho^{-1}\right)\right)\right|. %=(-1)^n
%P_{n}^{(\alpha,\beta)}\left(-\tfrac{1}{2}\left(\rho+\rho^{-1}\right)\right).
\end{equation}

\item[\rm (iii)] If $\alpha=\beta \geq -1/2$, then the maximum value
of $\left|P_{n}^{(\alpha,\beta)}(z)\right|$ is attained at two endpoints of the major axis, i.e.,
\begin{equation}\label{eq:maxCaseIII}
\max_{z\in \mathcal{E}_{\rho}} \left|P_{n}^{(\alpha,\beta)}(z)\right| %=
%P_{n}^{(\alpha,\beta)}\left(\tfrac{1}{2}\left(\rho+\rho^{-1}\right)\right)
= \left|P_{n}^{(\alpha,\beta)}\left(\pm
\tfrac{1}{2}\left(\rho+\rho^{-1}\right)\right)\right|.
%=(-1)^n
%P_{n}^{(\alpha,\beta)}(-\tfrac{1}{2}(\rho+\rho^{-1})).
\end{equation}
Moreover, the maximum value can only be attained at these two real
points $\pm \tfrac{1}{2}\left(\rho+\rho^{-1}\right)$ if $\alpha=\beta>-1/2$.
\end{enumerate}
\end{theorem}

The assertion in item $\rm (iii)$ corresponds to the case of Gegenbauer polynomials mentioned at the very beginning. Moreover, the condition $\alpha+\beta \geq -1$ implies that $\max\{\alpha,\beta\}\geq -\frac{1}{2}$. By setting $\rho=1$ in the above theorem, we recover the result concerning maximum of $\left|P_{n}^{(\alpha,\beta)}(x)\right|$ over the orthogonal interval $[-1,1]$, as explained after \eqref{eq:max}.

The proof of Theorem \ref{prop:location} relies on the explicit
formula of $P_{n}^{(\alpha,\beta)}(z)$ on the Bernstein ellipse
established in Theorem \ref{thm:jacobi ellipse}. The essential issue
here is to determine the signs of the coefficients $d_{k,n}$
appearing in \eqref{eq:jacobi ellipse} under various conditions on
the parameters $\alpha$ and $\beta$; see Proposition \ref{prop:sign}
below. To proceed, we start with the following proposition which
reveals a recurrence relation for the coefficients $\{ d_{k,n}
\}_{k=0}^{n}$ and plays a fundamental role in the sequel.
\begin{proposition}\label{prop:recurr}
With $d_{k,n}$ defined in \eqref{eq:dkn}, we have, for each $k\geq
0$ and $k+2 \leq n$,
\begin{align}\label{eq:recurrence} d_{k,n} &=
\frac{2(\alpha-\beta)(k+1)}{ n(n+\alpha+\beta+1) - k^2 - (\alpha+\beta+1)k } d_{k+1,n} \nonumber \\
&~~~  +  \frac{  n(n+\alpha+\beta+1) - (k+2)^2 +
(\alpha+\beta+1)(k+2) }{ n(n+\alpha+\beta+1) - k^2 -
(\alpha+\beta+1) k } d_{k+2,n},
\end{align}
with initial conditions
\begin{equation}\label{eq:initialcondition}
d_{n,n} =
\frac{\Gamma(2n+\alpha+\beta+1)}{2^{2n}\Gamma(n+\alpha+\beta+1)\Gamma(n+1)},
~~ d_{n-1,n} = \frac{(\alpha-\beta)
\Gamma(2n+\alpha+\beta)}{2^{2n-1}\Gamma(n+\alpha+\beta+1)\Gamma(n)}.
\end{equation}
%\begin{align}
%&[ n(n+\alpha+\beta+1) -k^2 - (\alpha+\beta+1)k ] d_k^n +
%2(\beta-\alpha)(k+1) d_{k+1}^n \nonumber \\
%&~~~ - [ n(n+\alpha+\beta+1) - (k+2)^2 + (\alpha+\beta+1) (k+2) ]
%d_{k+2}^n = 0,
%\end{align}
\end{proposition}
\begin{proof}
In view of \eqref{eq:jacobi and cheby}, it is readily seen that
\begin{align}\label{eq:JacoCheb}
P_n^{(\alpha,\beta)}(x) = d_{0,n} + 2\sum_{k=1}^{n} d_{k,n} T_k(x).
\end{align}
We recall from \cite[Theorem~4.2.1]{szego1939orthogonal} that the
Jacobi polynomial $P_n^{(\alpha,\beta)}(x)$ satisfies the following
linear differential equation
\[
(1-x^2) y''(x) + [\beta - \alpha - (\alpha+\beta+2)x] y'(x) +
n(n+\alpha+\beta+1) y(x) = 0.
\]
Substituting \eqref{eq:JacoCheb} into the above equation gives
\begin{align}\label{eq:DiffEq}
&2 \sum_{k=1}^{n} d_{k,n} \big\{ (1-x^2)T_k^{''}(x) +
[\beta-\alpha-(\alpha+\beta+2)x ] T_k'(x) \nonumber \\
&~~~ + n(n+\alpha+\beta+1) T_k(x) \big\}  + n(n+\alpha+\beta+1)
d_{0,n} = 0.
\end{align}

Our strategy now is to rewrite the left hand side of
\eqref{eq:DiffEq} in terms of the Chebyshev polynomial of the second
kind $U_k(x)$. To this end, note that $T_k(x)$ satisfies
\[
(1-x^2)T_k''(x) - xT_k'(x) + k^2 T_k(x) = 0,
\]
we then obtain from \eqref{eq:DiffEq} that
\begin{align*}
&2 \sum_{k=1}^{n} d_{k,n} \big\{
[\beta-\alpha-(\alpha+\beta+1)x ] T_k'(x)  \\
&~~~ + [n(n+\alpha+\beta+1)-k^2] T_k(x) \big\}  +
n(n+\alpha+\beta+1) d_{0,n} = 0.
\end{align*}
This, together with the facts (cf. \cite[\S 18.9]{DLMF}) that
\begin{align*}
T_k'(x) &= k U_{k-1}(x),  \\
2xU_k(x)& = U_{k+1}(x)+U_{k-1}(x), \quad  k\geq1, \\
2T_k(x)&=U_k(x)-U_{k-2}(x), ~\textrm{$k\geq 1$ with $U_{-1}(x)=0$,}
\end{align*}
implies
\begin{align*}
& \sum_{k=1}^{n} d_{k,n} \big\{
[n(n+\alpha+\beta+1)-k^2-k(\alpha+\beta+1) ] U_k(x)+2(\beta-\alpha)kU_{k-1}(x) \nonumber \\
&~~~ - [n(n+\alpha+\beta+1)-k^2+k(\alpha+\beta+1)] U_{k-2}(x) \big\}
+ n(n+\alpha+\beta+1) d_{0,n} = 0.
\end{align*}
By setting the coefficients of $U_{k}(x)$, $1\leq k \leq n-2$ and
the constant term to be zero, the recurrence relation
\eqref{eq:recurrence} is immediate.

This completes the proof of Proposition \ref{prop:recurr}.
\end{proof}

\begin{remark}
From \eqref{eq:recurrence}, it is readily seen that if
$\alpha=\beta$, the three-term recurrence relation can be simplified
as
\begin{align}
d_{k,n} = \frac{ n(n+2\alpha+1) - (k+2)^2 + (2\alpha+1)(k+2) }{
n(n+2\alpha+1) - k^2 - (2\alpha+1) k } d_{k+2,n}.
\end{align}
In addition, note that the coefficients $d_{k,n}$ in \eqref{eq:dkn} involve the hypergeometric function ${}_3\mathrm{F}_2$, it would be helpful to use the recurrence relation \eqref{eq:recurrence} in actual computations.
\end{remark}

%\begin{figure}[ht]
%\centering
%\includegraphics[width=4.7cm]{expansioncoeff0.eps}
%\includegraphics[width=4.7cm]{expansioncoeff1.eps}~
%\includegraphics[width=4.7cm]{expansioncoeff2.eps}
%\caption{Relative errors of the computed values of
%$\{d_{k,n}\}_{k=0}^{n}$ by the recurrence relation
%\eqref{eq:recurrence} for three different sets of parameters
%$\alpha,\beta$ with $n=80$.} \label{fig:Relative error}
%\end{figure}
%
%
%
%To testify the accuracy of computations of $\{d_{k,n}\}_{k=0}^{n}$
%via the recurrence relation \eqref{eq:recurrence}, we present
%several numerical experiments on the relative errors for three
%different pairs of parameters $(\alpha,\beta)$. The computations
%were performed in {\sc Maple} with $16$ significant digits. Note
%that the direct evaluation of the first two initial values $d_{n,n}$
%and $d_{n-1,n}$ may result in an overflow for large $n$. This issue
%can be overcome by reformulating both values via logarithms. The
%relative errors are displayed in Figure \ref{fig:Relative error},
%where the exact values of $\{d_{k,n}\}_{k=0}^{n}$ were evaluated by
%\eqref{eq:dkn} with $200$ significant digits. Clearly, the accuracy
%of computations based on the recurrence relation
%\eqref{eq:recurrence} is quite satisfactory.

As a consequence of Proposition \ref{prop:recurr}, we are able to determine the signs of the
coefficients $\{d_{k,n} \}_{k=0}^{n}$ in the following proposition.
\begin{proposition}\label{prop:sign}
For $0\leq k \leq n$, we have
\begin{enumerate}
\item[\rm (i)] If $\alpha>\beta$ and $\alpha+\beta\geq-1$, then $d_{k,n}>0$.

\item[\rm (ii)] If $\alpha<\beta$ and $\alpha+\beta\geq-1$, then
$(-1)^{n-k}d_{k,n}>0$.

\item[\rm (iii)] If $\alpha=\beta>-\frac{1}{2}$, then
\begin{equation}\label{eq:dkn1}
d_{k,n}\left\{
         \begin{array}{ll}
           >0, & \hbox{if $n-k$ is even,} \\[8pt]
           =0, & \hbox{if $n-k$ is odd.}
         \end{array}
       \right.
\end{equation}
If $\alpha=\beta=-\frac{1}{2}$, then
 \begin{equation}\label{eq:dkn2}
d_{k,n}\left\{
         \begin{array}{ll}
           >0, & \hbox{if $k=n$,} \\ [8pt]
           =0, & \hbox{if $k=1,2,\ldots,n-1$.}
         \end{array}
       \right.
\end{equation}
\end{enumerate}
\end{proposition}
\begin{proof}
If $\alpha>\beta$ and $\alpha+\beta\geq-1$, it is easily seen that
\[
\frac{2(\alpha-\beta)(k+1)}{ n(n+\alpha+\beta+1) - k^2 -
(\alpha+\beta+1)k }=\frac{2(\alpha-\beta)(k+1)}{n^2-k^2+(\alpha+\beta+1)(n-k)}>0,
\]
and
\begin{multline}
\frac{  n(n+\alpha+\beta+1) - (k+2)^2 + (\alpha+\beta+1)(k+2) }{
n(n+\alpha+\beta+1) - k^2 - (\alpha+\beta+1) k }
\\
=\frac{  n^2 - (k+2)^2 + (\alpha+\beta+1)(n+k+2)}{
n^2-k^2+(\alpha+\beta+1)(n-k)} >0,
\end{multline}
for $0\leq k \leq n-2$.
These, together with the recurrence relation \eqref{eq:recurrence} and the fact that both of the
initial values $d_{n-1,n}$ and $d_{n,n}$ are positive (see \eqref{eq:initialcondition}), imply the
assertion in item $\rm (i)$.

Similarly, if $\alpha<\beta$ and $\alpha+\beta\geq-1$, we have
\[
\frac{2(\alpha-\beta)(k+1)}{ n(n+\alpha+\beta+1) - k^2 -
(\alpha+\beta+1)k }<0,
\]
and
\[
\frac{  n(n+\alpha+\beta+1) - (k+2)^2 + (\alpha+\beta+1)(k+2) }{
n(n+\alpha+\beta+1) - k^2 - (\alpha+\beta+1) k } >0,
\]
for $0\leq k \leq n-2$. Since $d_{n-1,n}<0$ and $d_{n,n}>0$ in this case, we again obtain from \eqref{eq:recurrence}
that $\{d_{k,n}\}_{k=0}^{n}$ is an alternating sequence, as required.

Finally, if $\alpha=\beta\geq -\frac{1}{2}$, the assertion in item $\rm (iii)$
follows immediately from \eqref{eq:GegenbauerCoeff}.

This completes the proof of Proposition \ref{prop:sign}.
\end{proof}

We are now ready to prove Theorem \ref{prop:location}.

\paragraph{Proof of Theorem \ref{prop:location}}
If $\alpha>\beta$ and $\alpha+\beta\geq-1$, we recall from item $\rm (i)$ in
Proposition \ref{prop:sign} that $d_{k,n}>0$ for $0\leq k \leq n$. On account of \eqref{eq:para} and
\eqref{eq:jacobi ellipse}, it is straightforward to see that
\[
 \left|P_{n}^{(\alpha,\beta)}(z)\right| = \left| \sum_{k = -n}^{n} d_{|k|,n} u^k \right| \leq \sum_{k = -n}^{n} d_{|k|,n}
 \rho^k = P_{n}^{(\alpha,\beta)}\left(\tfrac{1}{2}\left(\rho+\rho^{-1}\right)\right).
\]
Thus, $\max_{z\in \mathcal{E}_{\rho}} \left|P_{n}^{(\alpha,\beta)}(z)\right|$ can be achieved if and only if $u=\rho$, which is
\eqref{eq:maxcaseI}.

Similarly, if $\alpha<\beta$ and $\alpha+\beta\geq-1$, a combination of item $\rm (ii)$ in
Proposition \ref{prop:sign} and \eqref{eq:jacobi ellipse} implies that
\begin{multline*}
 \left|P_{n}^{(\alpha,\beta)}(z)\right| = \left| \sum_{k = -n}^{n} d_{|k|,n} u^k \right| \leq \sum_{k = -n}^{n} (-1)^{n-|k|} d_{|k|,n}
 \rho^k
\\= (-1)^n
 P_{n}^{(\alpha,\beta)}\left(-\tfrac{1}{2}\left(\rho+\rho^{-1}\right)\right)
=\left|P_{n}^{(\alpha,\beta)}\left(-\tfrac{1}{2}\left(\rho+\rho^{-1}\right)\right)\right|.
\end{multline*}
Hence, the maximum value can be achieved if and only if $u=-\rho$, as shown in \eqref{eq:maxCaseII}.

To show \eqref{eq:maxCaseIII}, we see from Remark \ref{rk:ultraspherical}, \eqref{eq:dkn1} and \eqref{eq:dkn2} that
\[
 \left|P_{n}^{(\alpha,\beta)}(z)\right| = \left| \sum_{k = 0}^{n} d_{|n-2k|,n} u^{n-2k} \right| \leq \sum_{k = 0}^{n} d_{|n-2k|,n}
 \rho^{n-2k}.
\]
Thus, the maximum value can be achieved when $u=\pm\rho$ and \eqref{eq:maxCaseIII} follows. Moreover,
if $\alpha=\beta>-\frac{1}{2}$, since $d_{|n-2k|,n}$ is strictly positive (see \eqref{eq:dkn1}) for $0\leq k \leq n$, the maximum value can only be achieved at two endpoints of the major axis.

This completes the proof of Theorem \ref{prop:location}.
\qed

%\begin{remark}
%Ismail in  proved the
%absolute value of the Gegenbauer polynomial attains its maximum
%value at the intersection point of the Bernstein ellipse with the
%positive real axis. Here we have extended the results to the Jacobi
%polynomials with $\alpha+\beta\geq-1$.
%\end{remark}

\begin{remark}
For the very special case $\alpha=\beta=-\frac{1}{2}$, the Jacobi polynomials
(up to a normalization constant) are the Chebyshev polynomials of the first kind $T_n(x)$; see \eqref{eq:Cheb}. In this case,  we obtain from \eqref{eq:ChebyEllipse} that for $z\in \mathcal{E}_{\rho}$,
\begin{align}\label{eq:ChebyFirstMod}
|T_n(z)| &= \frac{1}{2} \sqrt{\rho^{2n} + \rho^{-2n} +
2\cos(2n\theta)} \nonumber\\
&\leq \frac{1}{2} \sqrt{\rho^{2n} + \rho^{-2n} + 2} = \frac{1}{2}(\rho^n + \rho^{-n}).
\end{align}
It is then clear that $\max_{z\in \mathcal{E}_{\rho}} |T_n(z)|$ is attained if and only if
$\cos(2n\theta) = 1$, i.e., at $2n$ points
\begin{align}\label{eq:location upper bound}
\hat{z}_{n,k} = \frac{1}{2} \left( \rho e^{i \frac{k\pi}{n} } +
\left(\rho e^{i \frac{k\pi}{n} } \right)^{-1} \right), \quad
k=0,\ldots, 2n-1.
\end{align}
%This means that the maximum value $|T_n(z)|$ on the Bernstein ellipse can be attained at $2n$ points .
\end{remark}

%\BLUE{Q: What about the case $\alpha+\beta<-1$?}

\subsection{Minimum value}\label{sec:min}

The identification of minimum value of $\left|P_{n}^{(\alpha,\beta)}(z)\right|$ on the Bernstein ellipse $\mathcal{E}_{\rho}$
is, in general, much more involved than that of its maximum value. Since the minimum value will be zero when $\rho=1$, we will restrict our attention to $\rho>1$ and focus on the ultraspherical case $\alpha=\beta$
in this section. In view of the relations \eqref{eq:gegenbauerdef} and \eqref{eq:Cheb}, the results will be presented in terms of $T_n(x)$, $U_n(x)$ and $C_n^{\lambda}(x)$.

%In this section we analyze the minimum of the absolute value of
%Jacobi polynomials on the Bernstein ellipse. Note , we restrict our attention to the case
%$\rho>1$. We commence from investigating both kinds of Chebyshev
%polynomials and the results will provide important insight.

To provide some intuition about the location where Gegenbauer polynomials attain the minimum value, we perform some numerical experiments of $\left|C_n^{\lambda}(z)\right|$ with $z\in \mathcal{E}_{\rho}$; see Figures \ref{fig:Gegenbauer1}--\ref{fig:Gegenbauer3} for different choices of the parameters $\lambda$, $n$ and $\rho$. Note that we omit the numerical results for $-\frac{1}{2}<\lambda<0$ and even $n \geq 2$, since they are similar to those shown in Figure \ref{fig:Gegenbauer3}. The numerical studies imply that the minimum value depends on the parameters $\lambda$, $n$, $\rho$, and further suggest the following conjecture concerning the observations.
\begin{conjecture}
It is conjectured that
\begin{enumerate}
\item[\rm (i)] If $\lambda>0$ and $n\geq1$ is odd, $\min_{z\in \mathcal{E}_{\rho}} \left| C_n^\lambda(z)\right|$ is attained at $\pm\frac{i}{2}(\rho-\rho^{-1})$ for $\rho>1$, i.e., at two endpoints of the minor axis.

\item[\rm (ii)] If $\lambda>0$ and $n\geq2$ is even, there exists a critical value $\varrho(n,\lambda)$ depending on the parameters $n$ and $\lambda$ such that $\min_{z\in \mathcal{E}_{\rho}} \left| C_n^\lambda(z)\right|$ is attained at $\pm\frac{i}{2}(\rho-\rho^{-1})$ for $\rho \geq
\varrho(n,\lambda)$.

\item[\rm (iii)] If $-\frac{1}{2}<\lambda<0$ and $n\geq2$,
there exists a critical value $\widetilde\varrho(n,\lambda)$ depending on the parameters $n$ and $\lambda$ such that $\min_{z\in \mathcal{E}_{\rho}} \left| C_n^\lambda(z)\right|$ is attained at $\pm\frac{1}{2}(\rho+\rho^{-1})$ for $\rho\geq \widetilde\varrho(n,\lambda)$, i.e., at two endpoints of the major axis.
\end{enumerate}

\end{conjecture}

%Before proving the results, we perform some numerical experiments on
%the location of the minimum value of Gegenbauer polynomials on the
%Bernstein ellipse. In Figures \ref{fig:Gegenbauer1},
%\ref{fig:Gegenbauer2} and \ref{fig:Gegenbauer3}, we plot
%$|C_{n}^{\lambda}(z)|$ for different choices of the parameters
%$\lambda$, $n$ and $\rho$. Based on these numerical experiments, we
%can list several conjectures:

In what follows, we shall prove items $\rm (i)$ and $\rm (ii)$ of the above conjecture
under the assumptions that $\rho\geq \frac{1}{2}(\sqrt{2}+\sqrt{6})\approx 1.932$ and $\lambda\geq1$ (for item $\rm (ii)$).

\begin{figure}[ht]
\centering
\includegraphics[width=4.8cm]{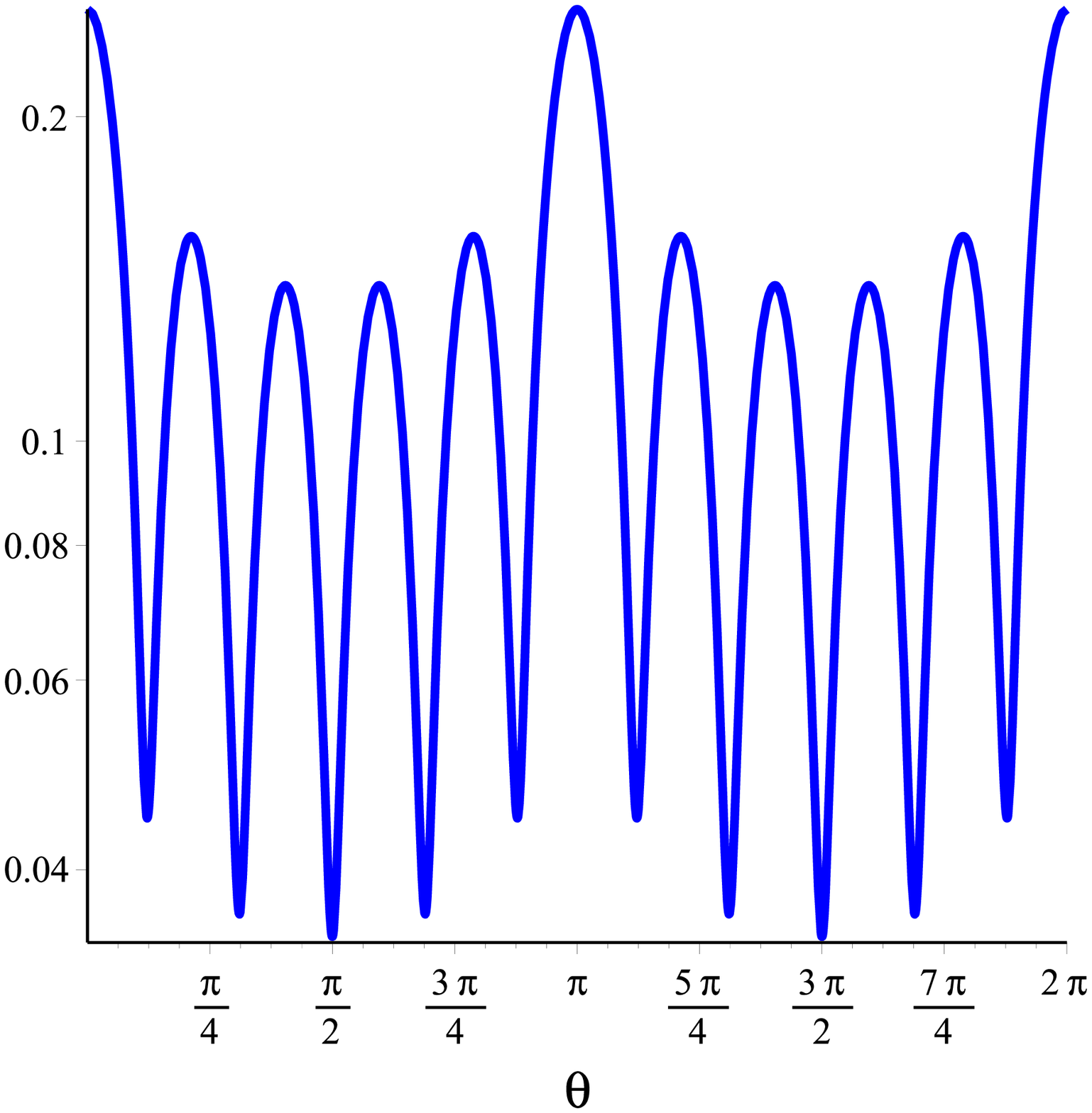}~~
\includegraphics[width=4.8cm]{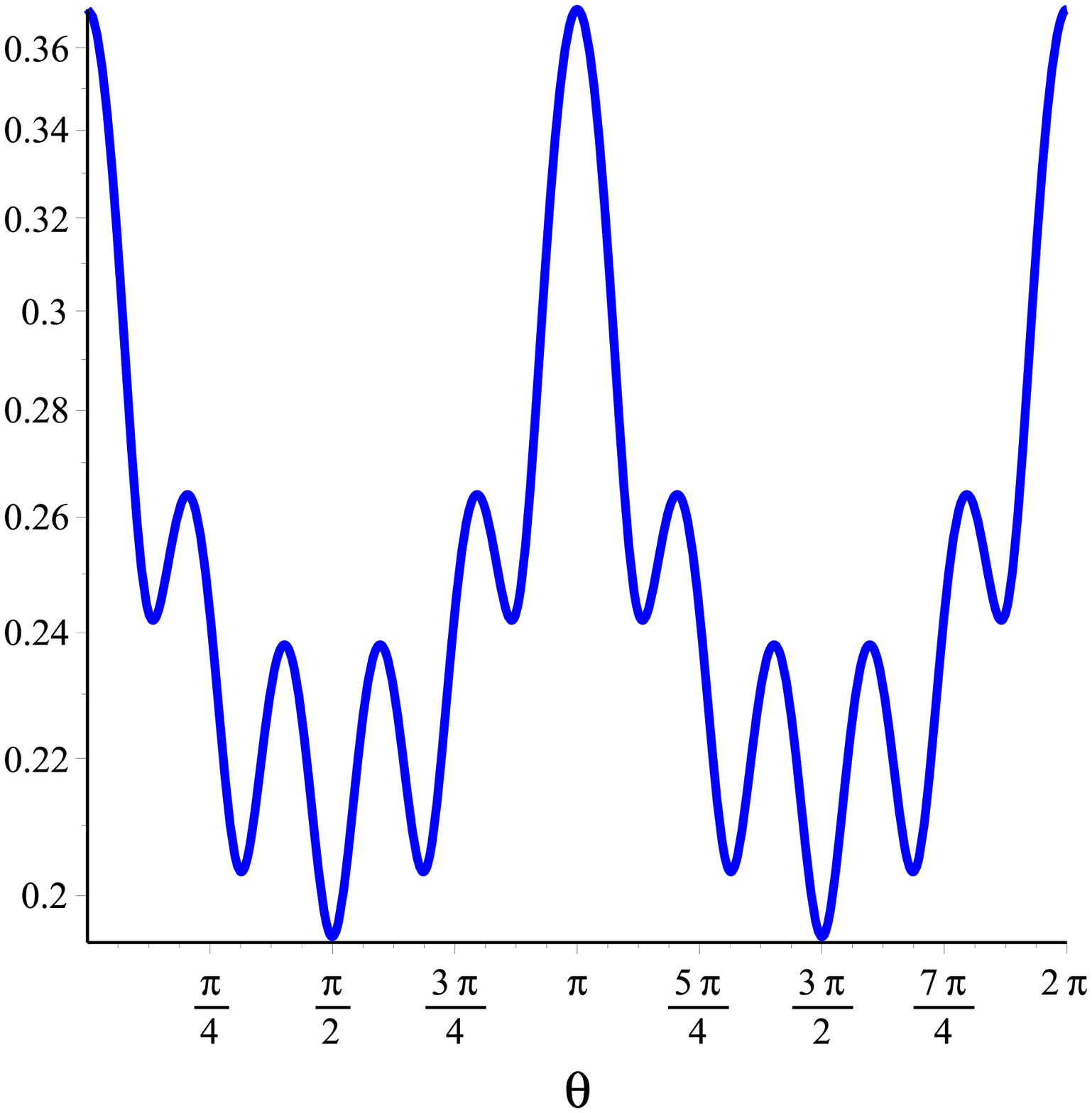}~~
\includegraphics[width=4.8cm]{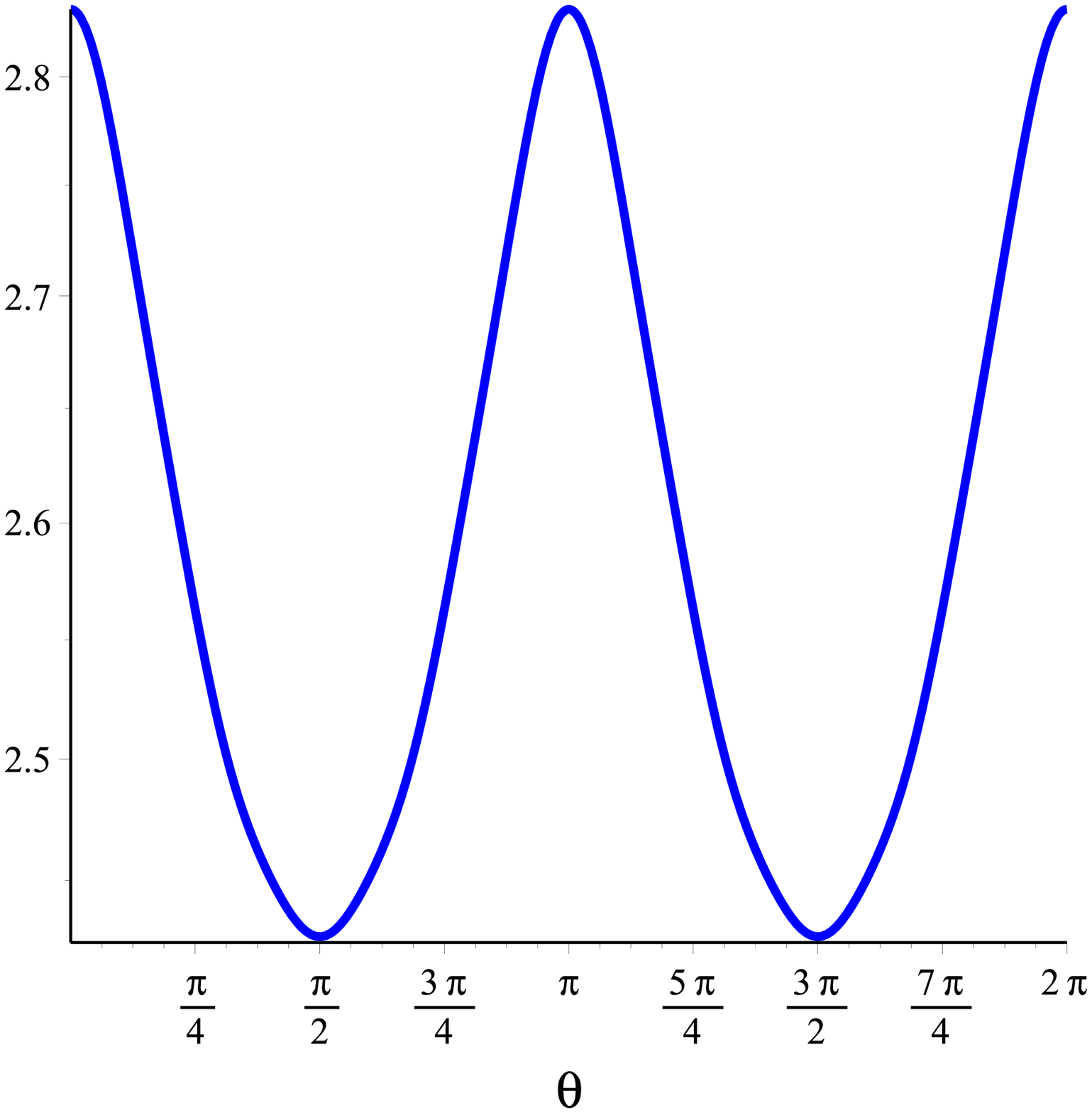}
\caption{Plot of
$\left|C_{5}^{1/4}(z)\right|$ with $z = \frac{1}{2}(\rho e^{i\theta} +
\rho^{-1} e^{-i\theta}) \in \mathcal{E}_{\rho}$ for $\rho=1.05$
(left), $\rho=1.25$ (middle) and $\rho=2$ (right). Here $\theta$
ranges from $0$ to $2\pi$. }\label{fig:Gegenbauer1}
\end{figure}

\begin{figure}[ht]
\centering
\includegraphics[width=4.8cm]{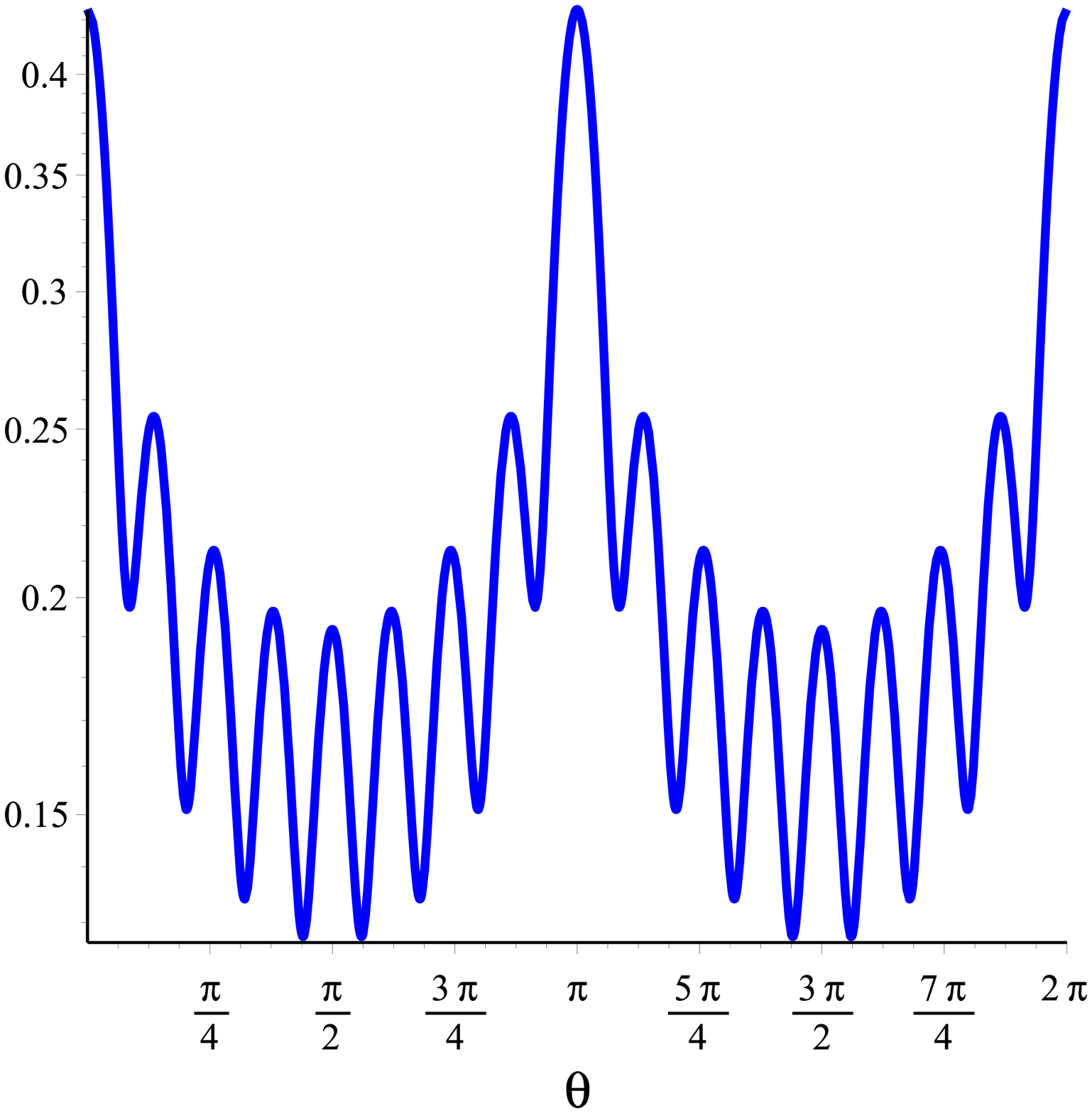}~~
\includegraphics[width=4.8cm]{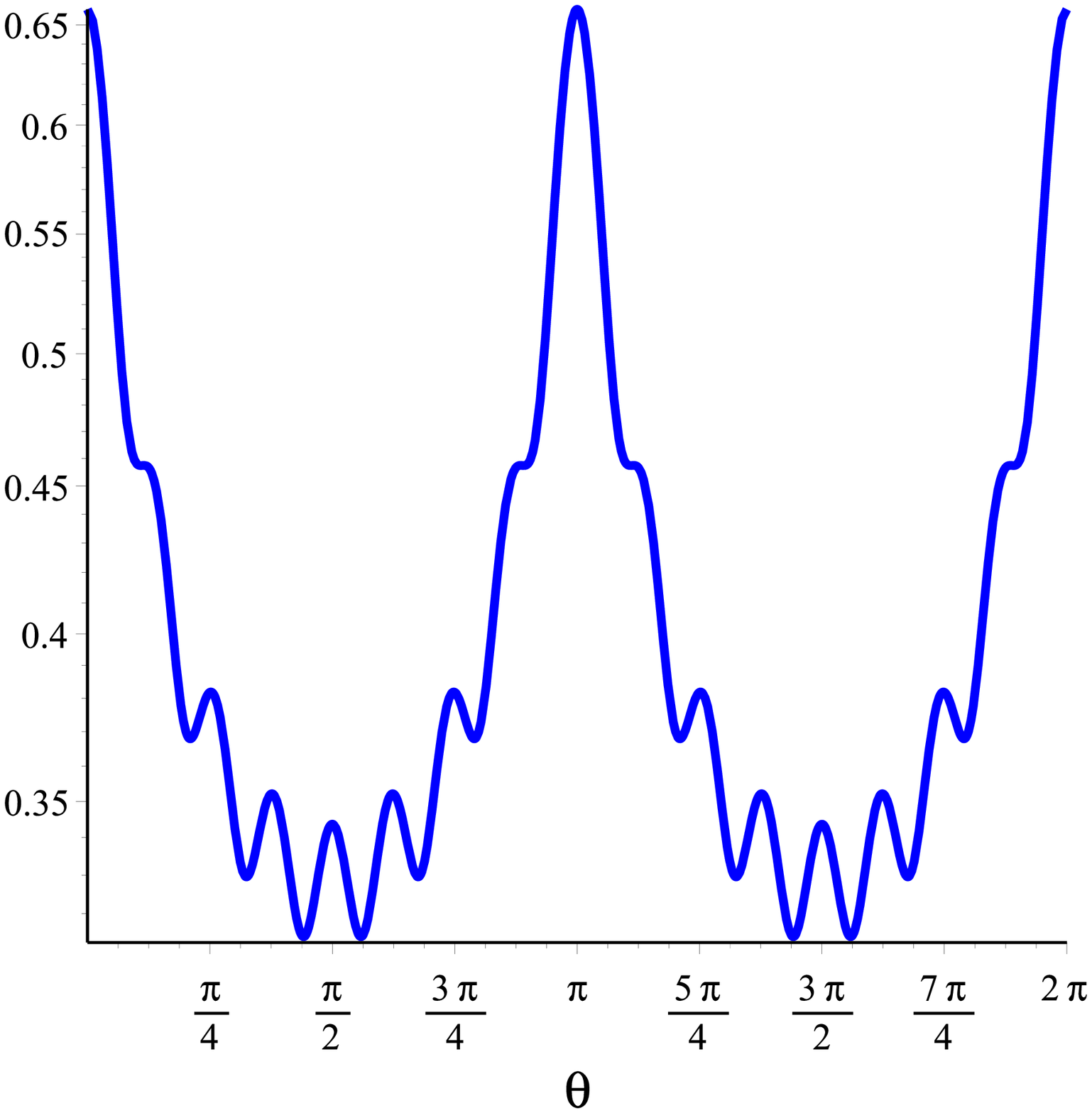}~~
\includegraphics[width=4.8cm]{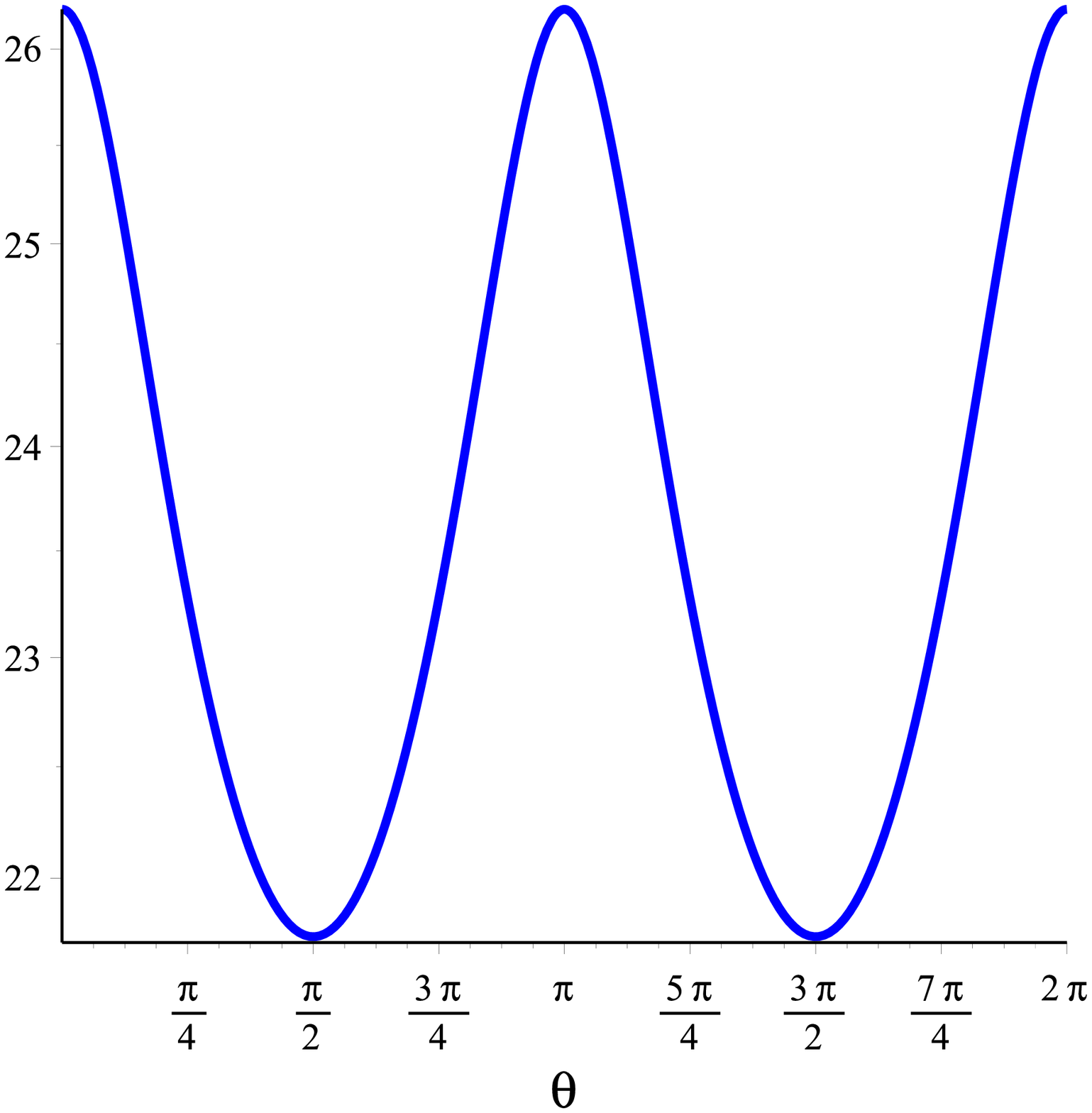}
\caption{Plot of
$\left|C_{8}^{1/3}(z)\right|$ with $z = \frac{1}{2}(\rho e^{i\theta} +
\rho^{-1} e^{-i\theta}) \in \mathcal{E}_{\rho}$ for $\rho=1.1$
(left), $\rho=1.2$ (middle) and $\rho=2$ (right). Here $\theta$
ranges from $0$ to $2\pi$. }\label{fig:Gegenbauer2}
\end{figure}

\begin{figure}[ht]
\centering
\includegraphics[width=4.8cm]{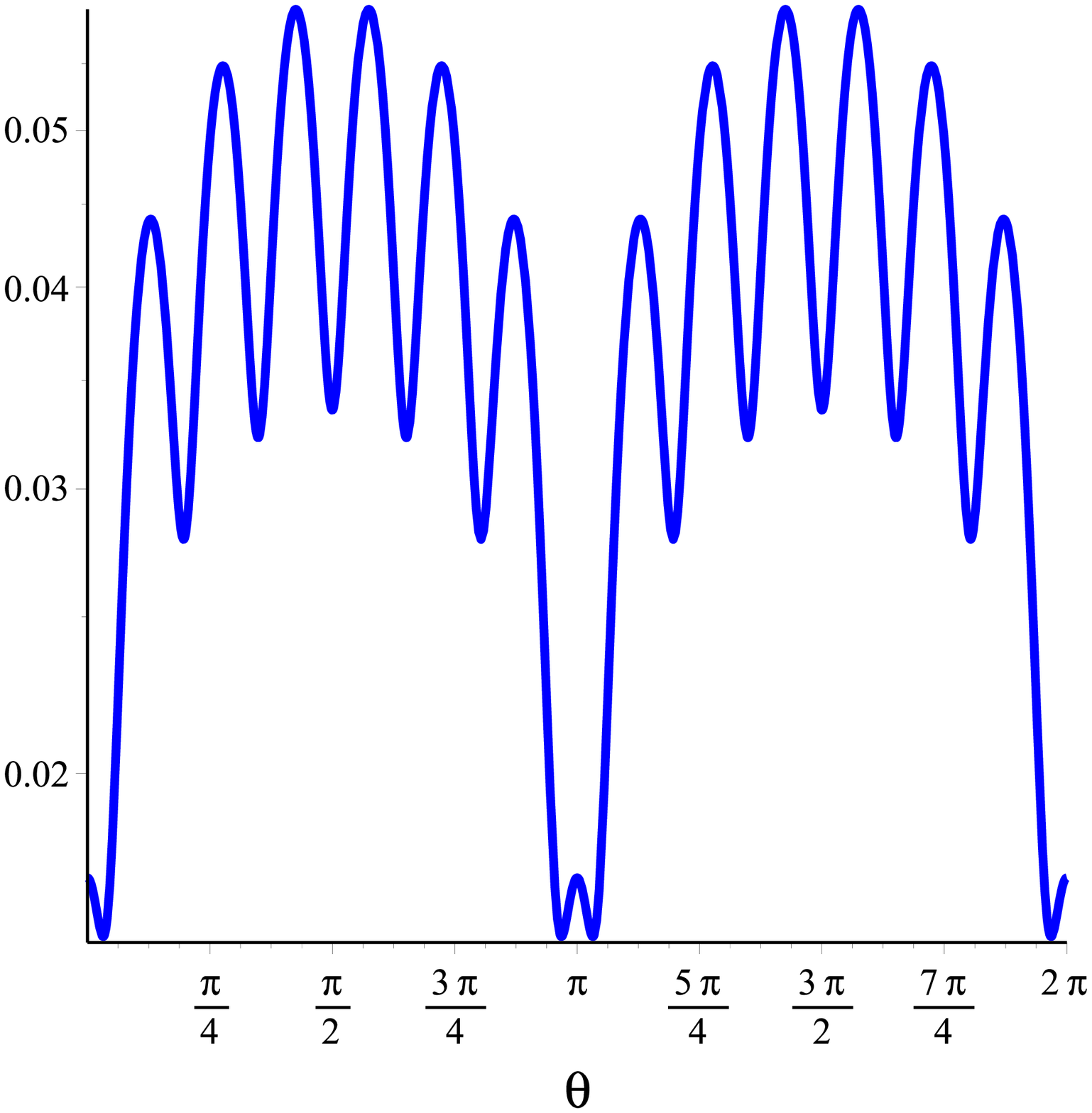}~~
\includegraphics[width=4.8cm]{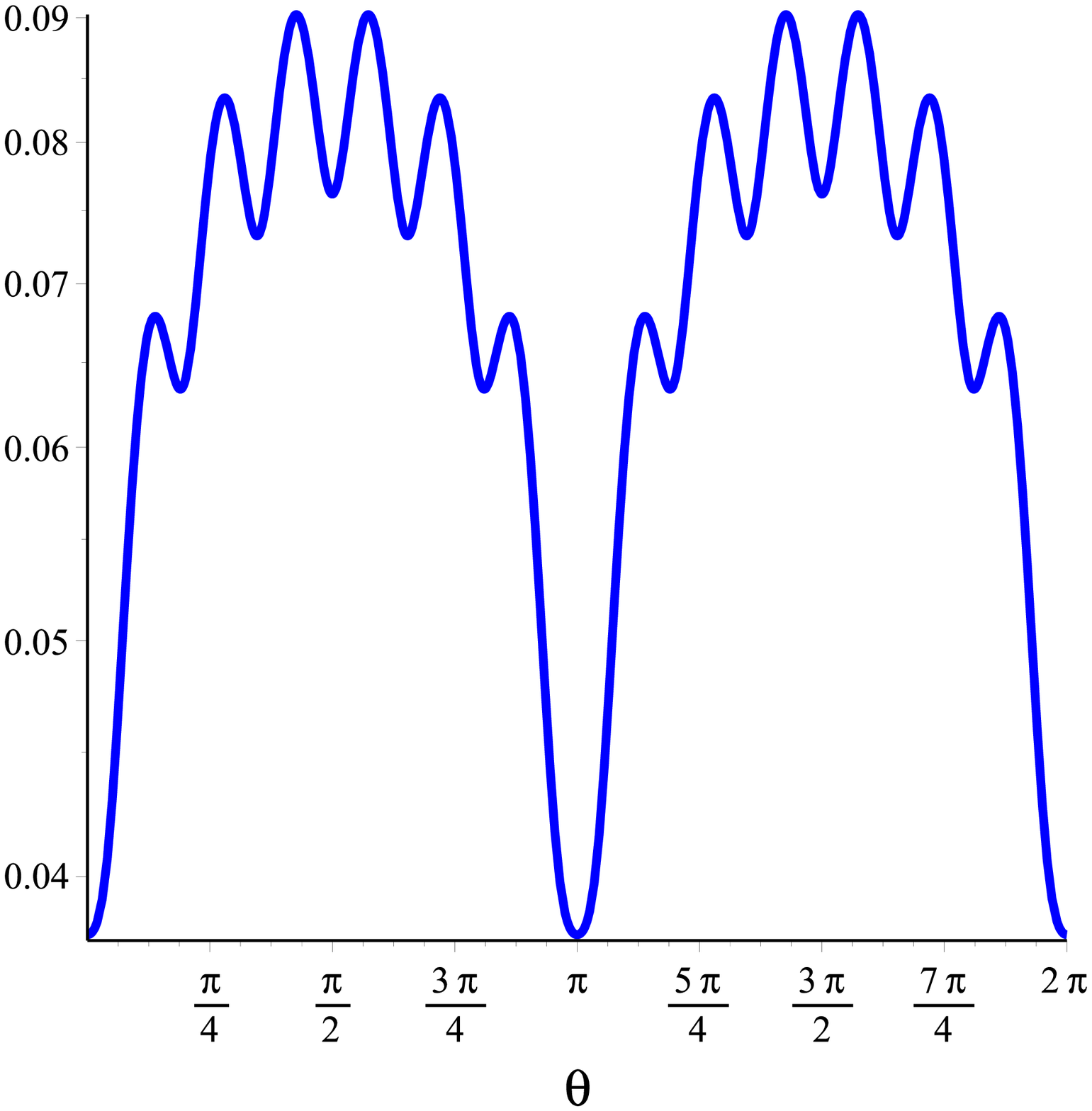}~~
\includegraphics[width=4.8cm]{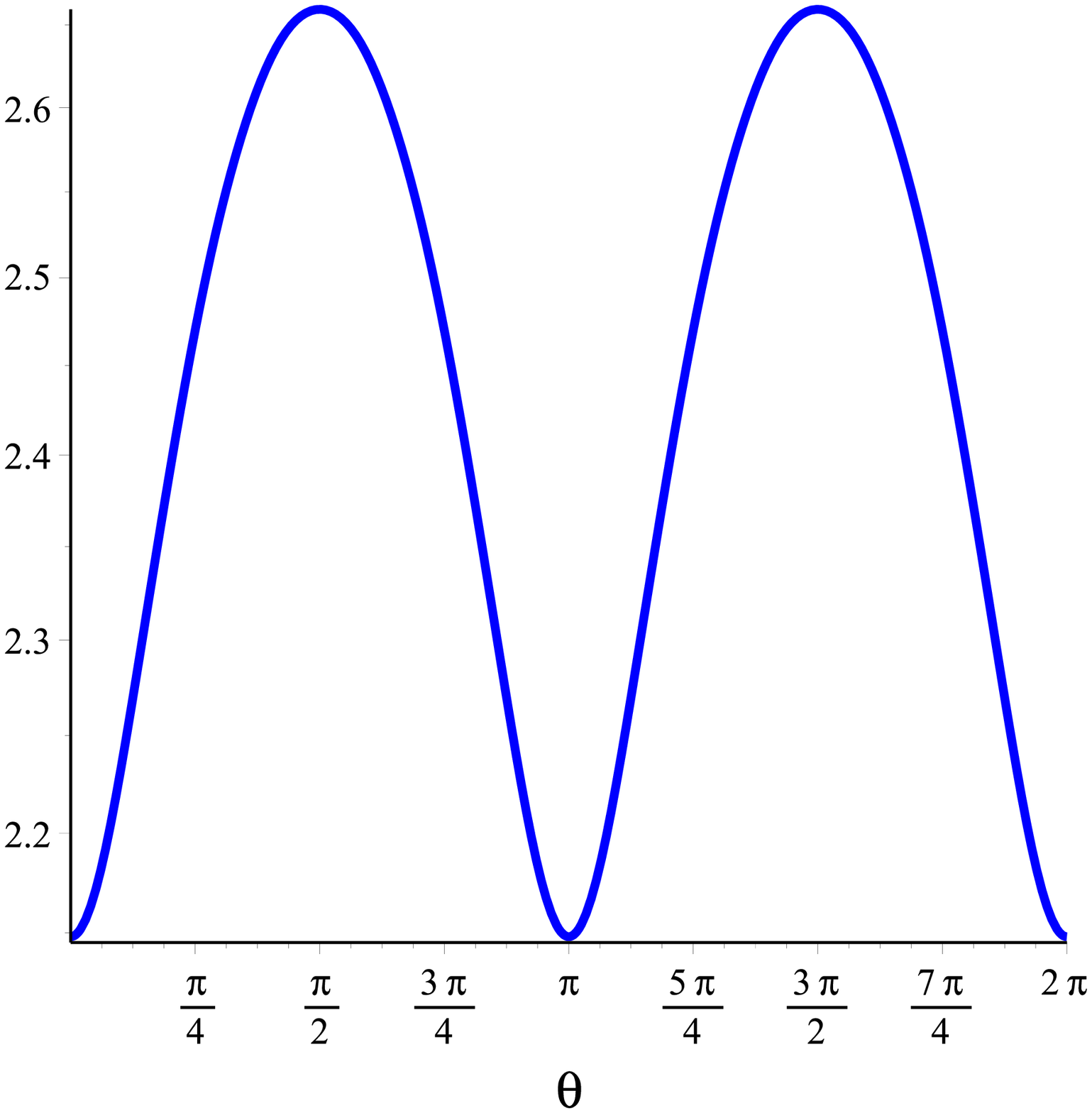}
\caption{Plot of
$\left|C_{7}^{-1/3}(z)\right|$ with $z = \frac{1}{2}(\rho e^{i\theta} +
\rho^{-1} e^{-i\theta}) \in \mathcal{E}_{\rho}$ for $\rho=1.1$
(left), $\rho=1.2$ (middle) and $\rho=2$ (right). Here $\theta$
ranges from $0$ to $2\pi$. }\label{fig:Gegenbauer3}
\end{figure}

We first deal with the Chebyshev polynomials of the first and second
kinds. The relevant results, on one hand, will provide important
insights for the general case, on the other hand they are crucial in
further analysis.

%\subsubsection{Gegenbauer polynomials}\label{sec:Chebyshev}
%We start from Chebyshev polynomials of the first kind, e.g.,
%$\lambda=0$.
\begin{theorem}[Minimum value of Chebyshev polynomials of the first kind $T_n(z)$]\label{thm:Cheby1kind}
For $\rho>1$ and $n \geq 1$, we have
\begin{align}
\min_{z\in \mathcal{E}_{\rho}} \left| T_n(z) \right| = \frac{1}{2}
\left(\rho^n - \rho^{-n} \right),
\end{align}
and the minimum value is attained at $2n$ points
\begin{align}\label{eq:location lower bound}
\check{z}_{n,k} = \frac{1}{2} \left( \rho e^{i \frac{2k+1}{2n}\pi }
+ \left(\rho e^{i \frac{2k+1}{2n}\pi } \right)^{-1} \right), \quad k
= 0,\ldots,2n-1.
\end{align}
\end{theorem}
\begin{proof}
From \eqref{eq:ChebyFirstMod} it is readily seen that $|T_n(z)|$ attains its
minimum value if and only if $$\cos(2n\theta) = -1,$$
i.e., for $\theta=\frac{2k+1}{2n}\pi$ with $k = 0,\ldots,2n-1$. The desired
results follow immediately.
\end{proof}

\begin{remark}\label{RK:ChebOdd}
From \eqref{eq:location lower bound}, it follows that, for odd $n$, the
minimum value of $T_n(z)$ can be attained at $\pm\frac{i}{2}(\rho-\rho^{-1}) $.
\end{remark}

We next proceed to the Chebyshev polynomial of the second kind $U_n(z)$. The following lower bound can be found
in \cite[formula (1.53)]{mason2003chebyshev}:
\begin{align}\label{eq:masonbound}
\left| U_n(z) \right| \geq \frac{\rho^{n+1} - \rho^{-n-1}}{\rho +
\rho^{-1}}, \quad z\in \mathcal{E}_{\rho}~~with~~ \rho>1.
\end{align}
Our next theorem shows that this lower bound is attainable
only when $n$ is odd. If $n$ is even, a new and attainable lower bound will be presented under some conditions on $\rho$.

\begin{theorem}[Minimum value of Chebyshev polynomials of the second kind $U_n(z)$]
\label{thm:ChebyUMin}
For $n\geq 1$, we have
\begin{align}\label{eq:lowoddcheby}
\min_{z\in \mathcal{E}_{\rho}} \left| U_n(z) \right|
%&=\left| U_n\left(\pm \frac{i}{2}\left(\rho-\rho^{-1}\right)\right) \right| \nonumber
%\\
&=\left\{
                                                        \begin{array}{ll}
                                                          {\displaystyle \frac{\rho^{n+1} - \rho^{-n-1}}{\rho + \rho^{-1}}}, & \hbox{if $n$ is odd and $\rho>1$,} \\[15pt]
                                                          {\displaystyle \frac{\rho^{n+1} + \rho^{-n-1}}{\rho + \rho^{-1}}},
& \hbox{if $n$ is even and $\rho\geq\rho_n^*$,}
                                                        \end{array}
                                                      \right.
\end{align}
where $\rho^{*}_n>1$ is the unique root of the equation
\begin{align}\label{eq:rhonequ}
 a_{n+1}(\rho) - (n+1) a_1(\rho) = 0,
\end{align}
and where
\begin{align}\label{def:akrho}
a_k(\rho) = \frac{1}{2} \left( \rho^k + \rho^{-k} \right), \qquad k\geq 0.
\end{align}
Moreover, in both cases the the minimum value is attained if and only if $z = \pm
\frac{i}{2}\left(\rho-\rho^{-1}\right)$, i.e., at two endpoints of the minor axis.
%The minimum of $\left| U_{n}(z) \right|$ depends on $\rho$ and the
%parity of $n$. If $n$ is an odd positive integer, then for $\rho>1$
%\begin{align}\label{eq:lowoddcheby}
%\min_{z\in \mathcal{E}_{\rho}} \left| U_n(z) \right| =
%\frac{\rho^{n+1} - \rho^{-n-1}}{\rho + \rho^{-1}},
%\end{align}
%and the minimum value is attained if and only if $z = \pm
%\frac{i}{2}\left(\rho-\rho^{-1}\right)$. If $n$ is an even positive
%integer, then for $\rho\geq\rho_n^{*}$
%\begin{align}\label{eq:lowereven}
%\min_{z\in \mathcal{E}_{\rho}} \left| U_n(z) \right| =
%\frac{\rho^{n+1} + \rho^{-n-1}}{\rho + \rho^{-1}},
%\end{align}

\end{theorem}

%e.g., $\lambda=1$. By \eqref{eq:ChebyEllipse}, we have
%\begin{equation}\label{def:varphi}
%|U_n(z)|^2 = \frac{a_{2n+2}(\rho) - \cos(( 2n+2 )\theta)}{ a_2(\rho)
%- \cos(2\theta)} := \varphi_n(\theta),
%\end{equation}
%where
%\begin{align}
%a_k(\rho) = \frac{1}{2} \left( \rho^k + \rho^{-k} \right), \quad
%k\geq 0,
%\end{align}
%and it is easy to see that $a_k(\rho)>1$ for $\rho>1$ and $k\geq1$.
%The following lower bound of $U_n(z)$
%In what follows, we shall show that

The above theorem can actually be seen from a remarkable connection
between $U_n(z)$ and the kernel $K_n(z)$ arising in the contour
integral representation of the remainder term of an $n$-point Gauss
quadrature for the Chebyshev weight function of the second kind.
More precisely, let $f$ be an analytic function on and within the
Bernstein ellipse $\mathcal{E}_{\rho}$. The Gaussian quadrature rule
for the Chebyshev weight function of the second kind reads
\begin{equation}
\int_{-1}^{1}f(t)(1-t^2)^{1/2}dt=\sum_{k=1}^n\lambda_k^{(n)}f\left(\tau_k^{(n)}\right)+R_n(f),
\end{equation}
where $\tau_k^{(n)}=\cos(k\pi/(n+1))$ are the zeros of the Chebyshev polynomial of the second kind $U_n(z)$, and $\lambda_k^{(n)}$ are the corresponding Christoffel numbers. The remainder term $R_n(f)$ admits the following contour representation:
\begin{equation}\label{eq:Rn}
R_n(f)=\frac{1}{2\pi i}\oint_{\mathcal{E}_{\rho}}K_n(z)f(z)dz,
\end{equation}
where the kernel $K_n(z)$ is given by
$$K_n(z)=\frac{q_n(z)}{U_n(z)}=\frac{\pi}{u^{n+1} U_n(z)},\quad z=\frac{1}{2}\left(u+u^{-1}\right), \quad |u|=\rho,$$
where $q_n(z)=\int_{-1}^{1}\frac{U_n(t)(1-t^2)^{1/2}}{z-t}dt$ and
the second equality follows from \cite[Equation 3.613.3]{Grad}. The
above formula particularly implies that $|U_n(z)|^2$ is proportional
to the reciprocal of $|K_n(z)|^2$ if $z\in \mathcal{E}_{\rho}$.
Since the maximum value of $|K_n(z)|$ over $\mathcal{E}_{\rho}$ has
been studied in the context of estimating the remainder term $R_n$
in \cite{gautschi83,gautschi1990note}, Theorem \ref{thm:ChebyUMin}
follows directly from \cite[Theorem 5.2]{gautschi83} and
\cite[Theorem 1]{gautschi1990note}. For completeness, we include a
more direct proof in what follows.

%In order to obtain the estimate of $R_n$ in \eqref{eq:Rn}, it is
%natural consider the maximum of $|K_n(z)|$ over
%$\mathcal{E}_{\rho}$. The investigation of this aspect was carried
%out in \cite{gautschi83,gautschi1990note}, which is based on the
%explicit formula (see \cite[formula (5.9)]{gautschi83})
%\begin{equation}\label{eq:Knonellipse}
%\left(\frac{\rho^{n+1} |K_n(z)|}{\pi}\right)^2=\frac{ a_2(\rho) -
%\cos(2\theta)}{a_{2n+2}(\rho) - \cos(( 2n+2 )\theta)},\quad
%z=\frac{1}{2} \left( \rho e^{i\theta} + \rho^{-1} e^{-i\theta}
%\right)\in\mathcal{E}_{\rho},
%\end{equation}
%where $a_k(\rho)$ is defined in \eqref{def:akrho}.
%
%The key observation here is, by \eqref{eq:ChebyEllipse}, it follows that
%\begin{equation}\label{def:varphi}
%|U_n(z)|^2 = \frac{a_{2n+2}(\rho) - \cos(( 2n+2 )\theta)}{ a_2(\rho)
%- \cos(2\theta)},\quad z=\frac{1}{2} \left( \rho e^{i\theta} + \rho^{-1} e^{-i\theta} \right)\in\mathcal{E}_{\rho}.
%\end{equation}
%Comparing the above formula with \eqref{eq:Knonellipse} shows that $|U_n(z)|^2$ is proportional to the reciprocal of $|K_n(z)|^2$. Thus, Theorem \ref{thm:ChebyUMin} follows from \cite[Theorem 5.2]{gautschi83} and \cite[Theorem 1]{gautschi1990note}. For convenience of the reader, we decide to include a more direct proof in what follows.

\paragraph{Proof of Theorem \ref{thm:ChebyUMin}}
By \eqref{eq:ChebyEllipse}, it is readily seen that
\begin{equation}\label{def:varphi}
|U_n(z)|^2 = \frac{a_{2n+2}(\rho) - \cos(( 2n+2 )\theta)}{ a_2(\rho)
- \cos(2\theta)},\quad z=\frac{1}{2} \left( \rho e^{i\theta} + \rho^{-1} e^{-i\theta} \right)\in\mathcal{E}_{\rho}.
\end{equation}
Denote by $\varphi_n(\theta)$ the function appearing on the right hand side of \eqref{def:varphi}. It is then equivalent to consider the minimum of $\varphi_n(\theta)$ for $\theta\in [0,2\pi]$.

We start with the easy case that $n$ is odd. By \eqref{def:varphi}, it is clear that
\[
\varphi_n(\theta) \geq  \frac{a_{2n+2}(\rho) - 1}{ a_2(\rho) +1} =
\left( \frac{\rho^{n+1} - \rho^{-n-1}}{\rho + \rho^{-1}} \right)^2,
\]
and the lower bound on the right hand side is attained if and only
if $\cos((2n+2)\theta)=1$ and $\cos(2\theta)=-1$, which gives
$\theta=\frac{\pi}{2}$ or $\frac{3}{2}\pi$. Thus, the minimum can only be
attained at two endpoints of the minor axis, i.e., at the points
$z=\pm\frac{i}{2}(\rho-\rho^{-1})$, as desired.

If $n$ is even, a straightforward calculation shows that
\begin{align} \label{eq:minevencase}
\varphi_n(\theta) -
\varphi_n\left(\frac{\pi}{2}\right) &= \frac{a_{2n+2}(\rho) - \cos((
2n+2 )\theta)}{ a_2(\rho) - \cos(2\theta)} - \frac{a_{2n+2}(\rho) + 1}{ a_2(\rho) + 1} \nonumber \\
%&= \frac{a_{2n+2}(\rho) - 2(\cos((
%n+1 )\theta))^2+1}{ a_2(\rho) - \cos(2\theta)} - \frac{a_{2n+2}(\rho) + 1}{ a_2(\rho) + 1} \nonumber \\
& = \frac{2(\cos\theta)^2}{a_2(\rho) - \cos(2\theta)} \left[
\frac{a_{2n+2}(\rho) + 1}{a_2(\rho) + 1} -
\left( \frac{\cos((n+1)\theta)}{\cos\theta} \right)^2 \right] \nonumber \\
& = \frac{2(\cos\theta)^2}{a_2(\rho) - \cos(2\theta)} \left[ \left(
\frac{a_{n+1}(\rho)}{a_1(\rho)} \right)^2 - \left(
\frac{\cos((n+1)\theta)}{\cos\theta} \right)^2  \right].
\end{align}
To this end, we note that, on one hand,
\begin{align*}
\left| \frac{\cos(n+1)\theta}{\cos\theta} \right| = \left|
\frac{\sin(n+1) (\frac{\pi}{2}-\theta)}{\sin (\frac{\pi}{2}-\theta)}
\right| = |U_n\left(t\right)|,
\end{align*}
where $t=\cos\left(\frac{\pi}{2}-\theta\right)$. It then follows from the property of $U_n(z)$ that
\begin{align}\label{eq:maxcos}
\max_{\theta\in [0,2\pi]} \left| \frac{\cos(n+1)\theta}{\cos\theta}
\right| = n+1,
\end{align}
and the upper bound can be attained if and only if $\theta=\frac{\pi}{2}$ or $\frac{3\pi}{2}$.
On the other hand, it is easily seen that the function $a_{n+1}(\rho)/a_1(\rho)$ is strictly increasing for
$\rho\in[1,\infty)$ and $n$ fixed. Hence, if $\rho\geq\rho_n^*$, we see from %\eqref{eq:rhonequ},
\eqref{eq:minevencase} and  \eqref{eq:maxcos} that
\[
\varphi_n(\theta) - \varphi_n\left(\frac{\pi}{2}\right) \geq 0.
\]
In addition, since
\[
\varphi_n\left(\frac{\pi}{2}\right) = \varphi_n\left(\frac{3\pi}{2}\right) = \frac{a_{2n+2}(\rho) +
1}{a_2(\rho) + 1} = \left( \frac{\rho^{n+1} + \rho^{-n-1}}{\rho +
\rho^{-1}} \right)^2,
\]
the second case in $\eqref{eq:lowoddcheby}$ follows.
It is also easy to see that the minimum is attained if and only if $\theta=\frac{\pi}{2}$
or $\frac{3\pi}{2}$.

This completes the proof of Theorem \ref{thm:ChebyUMin}.
\qed

\begin{remark}\label{rem:rhon}
For even $n$ and $1<\rho<\rho_n^{*}$, one can conclude from
\cite[Theorem~1]{gautschi1990note} that the minimum value of $\left|
U_n(z) \right|$ is attained at some $z^{*}=\frac{1}{2}(\rho
e^{i\theta^{*}} + \rho^{-1} e^{-i\theta^{*}})$ with $\theta^{*} \in
(\frac{n}{n+1} \frac{\pi}{2}, \frac{\pi}{2})$, which is slightly off
the imaginary axis. Moreover, from
\cite[Theorem~2]{gautschi1990note} we know that $\{ \rho_n^{*}
\}_{n=1}^{\infty}$ is a strictly decreasing sequence and $\rho_n^{*}
\rightarrow1$ as $n\rightarrow\infty$; see Figure \ref{fig:rho} for
an illustration.
\end{remark}

\begin{figure}[ht]
\centering
\includegraphics[width=7cm]{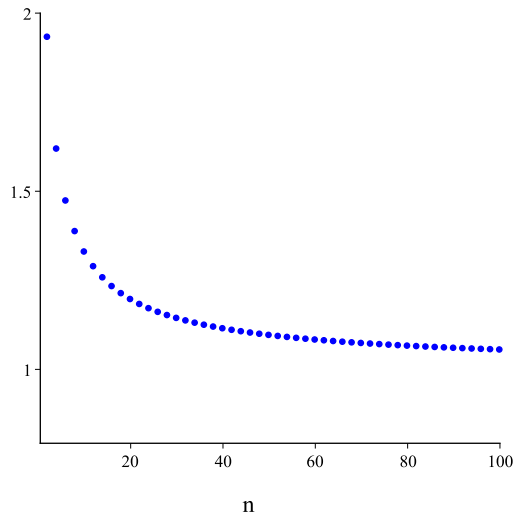}~~
\caption{Plot of the sequence $\{\rho_{n}^{*}
\}$ for $n=2,4,\ldots,100$.}
\label{fig:rho}
\end{figure}

We finally come to the Gegenbauer polynomials $C_n^{\lambda}(x)$.
Besides the trivial case\footnotemark[4] $n=1$, we have the
following theorem. \footnotetext[4]{If $n=1$,
$C_1^\lambda(x)=2\lambda x$. Thus, the minimum value of
$|C_1^\lambda(z)|$ can only be attained at two endpoints of the
minor axis $\pm \tfrac{i}{2}(\rho-\rho^{-1})$. }
\begin{theorem}[Minimum value of Gegenbauer polynomials $C_n^\lambda(z)$]\label{thm:Gegenbauer}
Let $\rho_2^{*}=\frac{1}{2}(\sqrt{2}+\sqrt{6})\approx 1.932$ be the
unique root of \eqref{eq:rhonequ} with $n=2$. For $\rho \geq
\rho_2^*$ and $n \geq 2$,  the minimum value of
$\left|C_n^\lambda(z)\right|$ is attained at two endpoints of the
minor axis, i.e.,
\begin{equation}\label{eq:minGegenbauer}
\min_{z\in \mathcal{E}_{\rho}} \left| C_n^\lambda(z)\right|= \left|C_n^{\lambda} \left(\pm \tfrac{i}{2}(\rho-\rho^{-1}) \right) \right|,
\end{equation}
provided $\lambda>1$, or $0<\lambda<1$ and $n$ is odd.
%Let $\lambda>1$.
%Then,
%\begin{align}
%\min_{z\in \mathcal{E}_{\rho}} \left| C_n^{\lambda}(z) \right| =
%\left|C_n^{\lambda} \left(\pm \tfrac{i}{2}(\rho-\rho^{-1}) \right)
%\right|,
%\end{align}
%and this is true for $\rho>1$ when $n=1$ and for $\rho\geq
%\rho_2^{*}$ when $n\geq2$.
\end{theorem}

We precede the proof of Theorem \ref{thm:Gegenbauer} with the following lemma.
\begin{lemma}\label{lemma:MinR}
Let $z\in \mathcal{E}_{\rho}$ and
define
\[
\mathcal{R}(z) = \frac{z^2-s^2}{z^2-t^2},
\]
where $s,t\in(0,1)$. Then, for $s>t$ and $\rho \geq \rho_2^{*}$,
\begin{align}
\max_{z\in \mathcal{E}_{\rho}} \left| \mathcal{R}(z) \right| =
\left| \mathcal{R}\left( \pm \tfrac{i}{2}(\rho-\rho^{-1}) \right)
\right|.
\end{align}
\end{lemma}
\begin{proof}
See \cite[Lemma~4.1]{schira1997remainder}.
\end{proof}

\paragraph{Proof of Theorem \ref{thm:Gegenbauer}}
Let $\{x_j^\lambda\}_{j=1}^n$ be the zeros of $C_n^{\lambda}(x)$ arranged in decreasing order. The symmetry relation \eqref{eq:symmetry relation} implies that
\begin{equation}\label{eq:Gegenxk}
C_n^{\lambda}(z) = k_n^{\lambda} \prod_{k=1}^{n} (z - x_k^{\lambda})
= k_n^{\lambda} z^{n - 2\floor{n/2}} \prod_{k=1}^{\floor{n/2}} (z^2
- (x_k^{\lambda})^2),
\end{equation}
where $k_n^{\lambda}$ is the leading coefficient of $C_n^{\lambda}(x)$. Moreover, we see
that $x_k^{\lambda} > 0$ for $k=1,\ldots,\floor{n/2}$.

%Suppose that $n\geq 2$ and $\lambda>1$.
Let $0< y_{\floor{n/2}} < \cdots<y_1<1$ be the positive zeros of the Chebyshev polynomials of the second kind $U_n(z)$. Again, we could rewrite $U_n(z)$ as
\begin{equation}\label{eq:Uny}
U_n(z) = 2^n z^{n - 2\floor{n/2}} \prod_{k=1}^{\floor{n/2}} (z^2 -
y_k^2).
\end{equation}
By \eqref{eq:decrease}, it follows that
\begin{equation}\label{eq:inequalityxy}
0<x_k^{\lambda}<y_k<1,
\end{equation}
for $k=1,\ldots,\floor{n/2}$.

To find $\min_{z\in \mathcal{E}_{\rho}} \left| C_n^\lambda(z)\right|$ is equivalent to find $\max_{z\in \mathcal{E}_{\rho}} \left|\frac{1}{ C_n^\lambda(z)}\right|$. A combination of \eqref{eq:Gegenxk} and \eqref{eq:Uny} gives
\begin{align}\label{eq:1Cn}
\left| \frac{1}{ C_n^{\lambda}(z)} \right| &=
\frac{2^n}{k_n^{\lambda}} \prod_{k=1}^{\floor{n/2}} \left| \frac{z^2
- y_k^2}{z^2 - (x_k^{\lambda})^2} \right| \times \left|
\frac{1}{U_n(z)} \right|.
\end{align}
In view of \eqref{eq:inequalityxy}, Lemma \ref{lemma:MinR}, Theorem \ref{thm:ChebyUMin} and the monotonicity of $\rho_n^*$ aforementioned in Remark \ref{rem:rhon},
we conclude that all the terms on the right hand side of \eqref{eq:1Cn} attain their
maximum values at $z=\pm\frac{i}{2}(\rho-\rho^{-1})$ for $\rho\geq \rho_2^*$. %Here, we have also made use of the fact that
%$\{ \rho_n^{*} \}_{n=1}^{\infty}$ is a strictly decreasing sequence; see \cite[Theorem 2]{gautschi1990note}.
Therefore, $|C_n^{\lambda}(z)|$ attains its minimum at two endpoints of the minor axis provided $\lambda>1$.

The case for $0<\lambda<1$ and odd $n$ can be proved in a similar
manner. We only need to replace $U_n(z)$ in \eqref{eq:1Cn} by the
Chebyshev polynomials of the first kind $T_n(z)$, and to make use of
Remark \ref{RK:ChebOdd} instead. The details are left to the
interested readers.

This completes the proof of Theorem \ref{thm:Gegenbauer}.\qed

\section{Asymptotic estimate of Jacobi polynomials on the Bernstein ellipse}
\label{sec:JacobiBound}

From the explicit formula \eqref{eq:XWZexpli} of Gegenbauer
polynomials on the Bernstein ellipse, the authors in
\cite{xie2013exponential} derive an asymptotic estimate of
$C_n^{\lambda}(z)$ as shown in \eqref{eq:estXWZ}. Due to the
complexity of the coefficients $d_{k,n}$ given in \eqref{eq:dkn}, it
is difficult to apply the same approach to obtain the asymptotic
estimate of Jacobi polynomials on the Bernstein ellipse.

To this end, we note that a more computable form has been given in
\cite{kuijlaars2004riemann}, where the authors actually consider
asymptotics of polynomials orthogonal with respect to a modified
Jacobi weight function
\begin{equation}\label{eq:modJac}
w(x) = (1-x)^{\alpha} (1+x)^{\beta} h(x),
\end{equation}
with $\alpha,\beta>-1$ and $h(x)$ being real analytic and strictly
positive on $[-1,1]$. Based on the Riemann-Hilbert (RH) approach \cite{Deift99book}, various asymptotics of the monic/orthonomal polynomials in the complex plane have been derived in \cite{kuijlaars2004riemann}, which in particular includes a full asymptotic expansion for the monic polynomials outside $[-1,1]$.

To state the relevant results, we need the function
\begin{equation}
\varphi(z)=z+\sqrt{z^2-1}, \qquad z \in \mathbb{C}\setminus [-1,1],
\end{equation}
where $\sqrt{z^2-1}$ is analytic in $\mathbb{C}\setminus [-1,1]$ and behaves like $z$ as $z\to \infty$. This function is a conforming mapping from $\mathbb{C}\setminus [-1,1]$ onto the exterior of the unit circle. Thus,
$$|\varphi(z)|>1, \qquad z \in \mathbb{C}\setminus [-1,1].$$
As in \cite{kuijlaars2004riemann}, we also define the Szeg\H{o} function of a weight $w$ by
\begin{equation}\label{def:D}
D(z)=D(z;w)=\exp\left(\frac{(z^2-1)^{1/2}}{2\pi}\int_{-1}^1\frac{\log w(x)}{\sqrt{1-x^2}}\frac{dx}{z-x}\right),\qquad z \in \mathbb{C}\setminus [-1,1],
\end{equation}
and
\begin{equation}\label{eq:Dinfty}
D_\infty=\lim_{z\to\infty}D(z)=\exp\left(\frac{1}{2\pi} \int_{-1}^1\frac{w(x)}{\sqrt{1-x^2}}dx\right).
\end{equation}
The function $D(z)$ is analytic for $z \in \mathbb{C}\setminus [-1,1]$.
%It is easy to check
%\begin{equation}
%D_+(x)D_-(x)=w(x), \qquad \textrm{a.e. $x\in(-1,1)$},
%\end{equation}
%where $D_{+(-)}(x)$ denotes the limiting value of $D(z)$ as $z$ tends to $x$ from above (below), and
%\begin{equation}
%D(z;w_1w_2)=D(z;w_1)D(z;w_2),
%\end{equation}
%for two weight functions $w_1$ and $w_2$.

Let $\pi_n(x)$ denote the monic orthogonal polynomial of degree $n$ associated with \eqref{eq:modJac}.
It is shown in \cite[Theorem~1.4]{kuijlaars2004riemann} that
$\pi_n(z)$ has an asymptotic expansion of the
form
\begin{align}\label{eq:monic asy}
\pi_n(z) \sim \frac{D_{\infty}}{D(z) }
\frac{\varphi(z)^{n+\frac{1}{2}}}{2^{n+\frac{1}{2}}
(z^2-1)^{\frac{1}{4}}}  \left[1 + \sum_{k=1}^{\infty}
\frac{\Pi_k(z)}{n^k} \right], \quad n\rightarrow\infty,
\end{align}
uniformly valid for $z$ in any compact subsets of
$\mathbb{C}\setminus [-1,1]$. The functions $\Pi_k(z)$, which are
analytic on $z\in \mathbb{C} \setminus [-1,1]$, are rational in
$\varphi$. They are explicitly computable via the RH approach but with more complicated form
as $k$ increases. The first two terms are
\begin{align}\label{def:Pi1}
\Pi_1(z) = - \frac{4\alpha^2-1}{8(\varphi(z)-1)} +
\frac{4\beta^2-1}{8(\varphi(z)+1)},
\end{align}
and
\begin{align}\label{def:Pi2}
\Pi_2(z) &= \frac{(4\alpha^2-1)(\alpha+\beta)}{16(\varphi(z)-1)} -
\frac{(4\beta^2-1)(\alpha+\beta)}{16(\varphi(z)+1)} -
\frac{(4\alpha^2-1)(4\beta^2-1)}{128(z^2-1)}\nonumber \\
&~~~ + \frac{2\alpha^2+2\beta^2-5}{64} \left[
\frac{4\alpha^2-1}{(\varphi(z)-1)^2} +
\frac{4\beta^2-1}{8(\varphi(z)+1)^2} \right].
\end{align}
For an efficient numerical calculations of the higher-order terms $\Pi_k(z)$, we refer to recent work \cite{DHO15}.

Obviously, the classical Jacobi polynomials correspond to the case $h(x) = 1$ in \eqref{eq:modJac}. We then have the following asymptotic estimate of Jacobi polynomials on the Bernstein ellipse in the variable of parametrization.
\begin{proposition}\label{prop:asyellip}
For $z\in \mathcal{E}_{\rho}$, i.e.,
\begin{equation}\label{eq:paraI}
z=\frac{1}{2} \left( u + u^{-1} \right),~~ u = \rho e^{i \theta},~~
\rho >1,~  ~0 \leq \theta  \leq   2\pi,
\end{equation}
we have, for large $n$,
\begin{align}\label{ineq:AsympEstim}
\left| \left( 1 - u^{-1} \right)^{-\alpha-\frac{1}{2}} \left( 1 +
u^{-1} \right)^{-\beta-\frac{1}{2}} - \frac{\sqrt{\pi n}
}{2^{\alpha+\beta} u^n} P_{n}^{(\alpha,\beta)}(z) \right| \leq
\Lambda(\rho,\alpha,\beta) n^{-1} + \mathcal{O}(n^{-2}).
\end{align}
where
\begin{align}\label{def:Pi1}
\Lambda(\rho,\alpha,\beta) = \max_{|u|=\rho} \left| \frac{4
\hat{\Pi}_1(u)-(\alpha+\beta)^2-(\alpha+\beta)-\frac{1}{2}}{4(1-u^{-1})^{\alpha+\frac{1}{2}}(1+u^{-1})^{\beta+\frac{1}{2}}}
\right|.
\end{align}
and
\begin{align}\label{def:Pi1}
\hat{\Pi}_1(u) = \frac{4\beta^2-1}{8(u+1)} -
\frac{4\alpha^2-1}{8(u-1)} .
\end{align}
Furthermore, the error is uniformly bounded for $z\in\mathcal{E}_{\rho}$ with $\rho>1$.
\end{proposition}

%\begin{proposition}\label{prop:asyellip}
%For $z\in \mathcal{E}_{\rho}$, i.e.,
%\begin{equation}\label{eq:para}
%z=\frac{1}{2} \left( u + u^{-1} \right),~~ u = \rho e^{i \theta},~~ \rho >1,~  ~0 \leq
%\theta  \leq   2\pi,
%\end{equation}
%we have, as $n\to \infty$,
%\begin{multline}\label{eq:asyElipse}
%P_{n}^{(\alpha,\beta)}(z) = \frac{ 2^{\alpha+\beta}u^{n+\alpha+\beta+1}}{\sqrt{\pi n}(u-1)^{\alpha+\frac{1}{2}}(u+1)^{\beta+\frac{1}{2}}} \\
%\times \left[1 +
%\frac{4
%\hat{\Pi}_1(u)-(\alpha+\beta)^2-(\alpha+\beta)-\frac{1}{2}}{4n} +\mathcal{O}\left(\frac{1}{n^2}\right)  \right],
%\end{multline}
%where
%\begin{align}\label{def:Pi1}
%\hat{\Pi}_1(u) = - \frac{4\alpha^2-1}{8(u-1)} +
%\frac{4\beta^2-1}{8(u+1)}
%\end{align}
%and the asymptotics is valid uniformly for $z\in
%\mathcal{E}_{\rho}$ with $\rho>1$.
%\end{proposition}
\begin{proof}
We first derive the uniform asymptotics of $P_{n}^{(\alpha,\beta)}(z)$ for $z\in \mathbb{C}\setminus [-1,1]$.
In view of the facts that
\begin{equation*}\label{eq:Szego}
D(z;(1-x)^\alpha(1+x)^{\beta})=\frac{(z-1)^{\alpha/2}(z+1)^{\beta/2}}{\varphi(z)^{(\alpha+\beta)/2}}
\end{equation*}
and
\begin{equation*}
D_\infty=\lim_{z\to\infty}D(z;(1-x)^\alpha(1+x)^{\beta})=2^{-(\alpha+\beta)/2},
\end{equation*}
it then follows from \eqref{eq:JacPol} and \eqref{eq:monic asy} that
\begin{align}\label{eq:asyJac1}
\frac{P_{n}^{(\alpha,\beta)}(z)}{k_n^{(\alpha,\beta)}} \sim \frac{
\varphi(z)^{n+\frac{\alpha+\beta+1}{2}}}{2^{n+\frac{\alpha+\beta+1}{2}}
(z-1)^{\frac{\alpha}{2}+\frac{1}{4}}
(z+1)^{\frac{\beta}{2}+\frac{1}{4}}  }  \left[1 +
\sum_{k=1}^{\infty} \frac{\Pi_k(z)}{n^k} \right], \quad n\rightarrow\infty,
\end{align}
uniformly valid for $z$ in any compact subsets of
$\mathbb{C}\setminus [-1,1]$, where $k_n^{(\alpha,\beta)}$ is
defined as in \eqref{eq:kn}.
Using asymptotic formulas for the Gamma functions (see
\cite[Formulas 5.11.3 and 5.11.13]{DLMF}), we deduce that
\begin{equation*}\label{eq:Kn2}
 k_n^{(\alpha,\beta)} = \frac{2^{n+\alpha+\beta}}{\sqrt{\pi n}}\left[1-\frac{(\alpha+\beta)^2+(\alpha+\beta)+\frac{1}{2}}{4n}+\mathcal{O}\left(\frac{1}{n^2}\right)\right].
\end{equation*}
This, together with \eqref{eq:asyJac1}, implies that
\begin{multline}\label{eq:asyJac2}
P_{n}^{(\alpha,\beta)}(z)=\frac{2^{\frac{\alpha+\beta}{2}}
\varphi(z)^{n+\frac{\alpha+\beta+1}{2}}}{\sqrt{2 \pi n}
(z-1)^{\frac{\alpha}{2}+\frac{1}{4}}
(z+1)^{\frac{\beta}{2}+\frac{1}{4}}  } \\
 \times \left[1 +\frac{4\Pi_1(z)-(\alpha+\beta)^2-(\alpha+\beta)-\frac{1}{2}}{4n} +\mathcal{O}\left(\frac{1}{n^2}\right) \right], \quad n\rightarrow\infty,
\end{multline}
where $\Pi_1(z)$ is given in \eqref{def:Pi1}.

If $z\in \mathcal{E}_{\rho} \subset \mathbb{C}\setminus [-1,1]$, which can be parameterized through the argument $u$ as shown in \eqref{eq:para}, it is straightforward to check that
\begin{align*}
\varphi(z)&=z+\sqrt{z^2-1}=u, \\
(z-1)^{\frac{\alpha}{2}+\frac{1}{4}}
(z+1)^{\frac{\beta}{2}+\frac{1}{4}}&=\frac{(u-1)^{\alpha+\frac{1}{2}}(u+1)^{\beta+\frac{1}{2}}}{(2u)^{\frac{\alpha+\beta+1}{2}}}.
\end{align*}
A combination of the above two formulas and \eqref{eq:asyJac2} gives
\begin{multline}\label{eq:asyElipse}
P_{n}^{(\alpha,\beta)}(z) = \frac{ 2^{\alpha+\beta}u^{n}}{\sqrt{\pi n}} (1-u^{-1})^{-\alpha-\frac{1}{2}}(1+u^{-1})^{-\beta-\frac{1}{2}} \\
\times \left[1 + \frac{4
\hat{\Pi}_1(u)-(\alpha+\beta)^2-(\alpha+\beta)-\frac{1}{2}}{4n}
+\mathcal{O}\left(\frac{1}{n^2}\right)  \right],
\end{multline}
where $\hat{\Pi}_1(u)$ is defined as in \eqref{def:Pi1} and the
asymptotics is valid uniformly for $z\in \mathcal{E}_{\rho}$ with
$\rho>1$. By using the above uniform asymptotic, it is
straightforward to derive the desired result \eqref{ineq:AsympEstim}
and this completes the proof of Proposition \ref{prop:asyellip}.
\end{proof}

\begin{remark}
One should compare the asymptotic estimate \eqref{ineq:AsympEstim} with \eqref{eq:estXWZ}. It is worthwhile to point out that the error in \eqref{ineq:AsympEstim} is of order $\mathcal{O}(1/n)$. Indeed, a full asymptotic expansion of $P_{n}^{(\alpha,\beta)}(z)$ in terms of powers of $1/n$ on the Bernstein ellipse $\mathcal{E}_{\rho}$ can be derived by combining \eqref{eq:asyJac1} and a full asymptotic expansion of the leading coefficient $k_n^{(\alpha,\beta)}$. We note that this form of asymptotic expansion has been mentioned in \cite[Theorem 8.21.9]{szego1939orthogonal}, but without explicit formulas for the coefficients.
\end{remark}

\begin{remark}
As a direct consequence of Proposition \ref{prop:asyellip}, we have
\begin{align}
\lim_{n\rightarrow\infty} \frac{P_{n}^{(\alpha,\beta) }(z)
\sqrt{n\pi}}{2^{\alpha+\beta} u^n} = \left(1 - u^{-1}
\right)^{-\alpha-\frac{1}{2}} \left(1 + u^{-1}
\right)^{-\beta-\frac{1}{2}},
\end{align}
where $z = \frac{1}{2}(u + u^{-1})$ and $|u|=\rho>1$.
\end{remark}

A further application of Proposition \ref{prop:asyellip} is that we are able to derive the following lower bound for the Jacobi polynomial on the Bernstein ellipse, which particularly implies a more explicit expression of the constant $C(\rho;\alpha,\beta)$ appearing in \eqref{eq:WZZ}.

\begin{corollary}
For $z=\frac{1}{2}(u+u^{-1})\in\mathcal{E}_{\rho}$, we have
\begin{align}\label{eq:lowerbound}
\min_{z\in\mathcal{E}_{\rho} } |P_{n}^{(\alpha,\beta)}(z)| \geq
 \frac{C_n(\alpha,\beta) 2^{\alpha+\beta} \pi^{-\frac{1}{2}}
\rho^{n}}{\displaystyle \max_{|u| = \rho} \left|
(1-u^{-1})^{\alpha+\frac{1}{2}} (1+u^{-1})^{\beta+\frac{1}{2}}
\right|\sqrt{n}},
\end{align}
where $C_n(\alpha,\beta)$ is a positive constant and $C_n(\alpha,\beta) \sim 1$ for large $n$. Moreover,
\begin{equation}\label{eq:maxovercircle}
\displaystyle \max_{|u| = \rho} \left|
(1-u^{-1})^{\alpha+\frac{1}{2}} (1+u^{-1})^{\beta+\frac{1}{2}}
\right| =\left\{
   \begin{array}{ll}
     (1+\rho^{-2})^{\alpha+\frac{1}{2}}, & \mbox{if $\alpha=\beta\geq -\frac{1}{2}$,}
     \\[8pt]
    (1-\rho^{-2})^{\alpha+\frac{1}{2}}, & \mbox{if $-1< \alpha=\beta< -\frac{1}{2}$.}
   \end{array}
 \right.
\end{equation}
\end{corollary}
\begin{proof}
The lower bound follows immediately from Proposition
\ref{prop:asyellip} and the elementary inequality $||z_1| -
|z_2||\leq |z_1-z_2|$.
To show \eqref{eq:maxovercircle}, by setting
$u=\rho e^{i\theta}$, $\rho>1$ and $0\leq \theta<2\pi$, it is easily
seen that
\begin{equation}
1-\rho^{-2} \leq
|1-u^{-2}|=\sqrt{1-2\rho^{-2}\cos(2\theta)+\rho^{-4}} \leq
1+\rho^{-2}.
\end{equation}
Hence, if $|u|=\rho$, $|1-u^{-2}|$ attains its maximum value at $\pm
\rho i$ and its minimum value at $\pm \rho $, which gives us \eqref{eq:maxovercircle}.
\end{proof}

\section{Concluding remarks}
In this paper, we have investigated several basic properties of
Jacobi polynomials on the Bernstein ellipse, which include the
explicit formula, extrema of the absolute values as well as a
refined asymptotic estimate. These results provide some further
insight into Jacobi polynomials and can be adaptable to some
practical applications such as establishing an explicit error bound
of the spectral interpolation at the Jacobi nodes.

%These results particularly extend those previously known only for
%some special cases and can be adaptable to some practical
%applications such as establishing an explicit bound of the spectral
%interpolation at the Jacobi nodes.

%Although the present
%research is inspired by error estimates for the maximum error of the
%Jacobi-Gauss type interpolation, we believe the relevant results
%will find both theoretical and practical uses in a wider range.

%One application of our results is to derive explicit error estimates for the maximum error of the Jacobi-Gauss type interpolation.
%Our results was used to derive explicit error estimates for
%the maximum error of the Jacobi-Gauss interpolation. We remark that,
%for other types of Jacobi-Gauss type interpolation such as
%Jacobi-Gauss-Lobatto interpolation, our result can also be used to
%obtain more explicit error estimates.

\section*{Acknowledgements}
Haiyong Wang thanks the hospitality of the School of Mathematical
Sciences at Fudan university where the present research was
initiated. His work was supported by the National Natural Science
Foundation of China under grant 11671160. Lun Zhang was partially
supported by The Program for Professor of Special Appointment
(Eastern Scholar) at Shanghai Institutions of Higher Learning (No.
SHH1411007), by National Natural Science Foundation of China (No.
11501120) and by Grant EZH1411513 from Fudan University.

%\appendix
%\section{Bound on $\frac{\cos(n+1)\theta}{\cos\theta}$}
%
%\begin{lemma}\label{lemma:bound cosine}
%For $n\geq2$ and is even, we have
%\begin{align}
%\max_{\theta\in [0,2\pi]} \left| \frac{\cos(n+1)\theta}{\cos\theta}
%\right| = n+1,
%\end{align}
%and the upper bound can be attained if and only if
%$\theta=\frac{\pi}{2}$ and $\theta=\frac{3\pi}{2}$.
%\end{lemma}
%\begin{proof}
%Note that
%\begin{align}
%\left| \frac{\cos(n+1)\theta}{\cos\theta} \right| = \left|
%\frac{\sin(n+1) (\frac{\pi}{2}-\theta)}{\sin (\frac{\pi}{2}-\theta)}
%\right| = |U_n\left(t\right)|,
%\end{align}
%where $t=\cos\left(\frac{\pi}{2}-\theta\right)$. The desired result
%follows by using the property of the Chebyshev polynomial of the
%second kind. The desired result follows.
%\end{proof}
%

%\bibliographystyle{abbrv}
%\bibliography{jacobi}

\end{document}